\documentclass[12pt,a4paper]{amsart} 
\usepackage{amsmath,amsfonts,amssymb,amsthm,amscd,slashed}
\usepackage{epsfig}
\pdfoutput=1

\usepackage{tikz}

\usepackage{xypic}

\usepackage{hyperref}
\usepackage{JK}                

\newtheorem{thm}{Theorem}%[section]

\newtheorem{lma}[thm]{Lemma} \newtheorem{prop}[thm]{Proposition}
\newtheorem{defn}[thm]{Definition}

\newtheorem{ex}[thm]{Example}
\newtheorem{rem}[thm]{Remark}

\newtheorem{ass}{Assumption}%[ass]

\def\1{\mathbb{I}}

\def\C{\mathbb{C}}
\DeclareMathOperator{\Cl}{Cl}

\DeclareMathOperator{\dom}{Dom}

\def\E{\mathcal{E}}

\def\into{\hookrightarrow}

\def\supp{\textup{supp}}

\def\T{\mathbb{T}}

\DeclareMathOperator{\tr}{Tr}

\def\tilde{\widetilde}

\title[Factorizing Dirac operators on almost-regular fibrations]{Factorization of Dirac operators on almost-regular fibrations of spin$^c$ manifolds}
\author{Jens Kaad and Walter D. van Suijlekom}

\address{Department of Mathematics and Computer Science, Syddansk Universitet, Campusvej 55, 5230, Odense M, Denmark}
\email{jenskaad@hotmail.com}
\address{Institute for Mathematics, Astrophysics and Particle Physics, Radboud University Nijmegen, Heyendaalseweg 135, 6525 AJ Nijmegen, The Netherlands}
\email{waltervs@math.ru.nl}

\date{\today}

 %\subjclass[2010]{19K35;}

\begin{document}

\begin{abstract}
We establish the factorization of the Dirac operator on an almost-regular fibration of spin$^c$ manifolds in unbounded KK-theory. As a first intermediate result we establish that any vertically elliptic and symmetric first-order differential operator on a proper submersion defines an unbounded Kasparov module, and thus represents a class in KK-theory. Then, we generalize our previous results on factorizations of Dirac operators to proper Riemannian submersions of spin$^c$ manifolds. This allows us to show that the Dirac operator on the total space of an almost-regular fibration can be written as the tensor sum of a vertically elliptic family of Dirac operators with the horizontal Dirac operator, up to an explicit `obstructing' curvature term. We conclude by showing that the tensor sum factorization represents the interior Kasparov product in bivariant K-theory. 
\end{abstract}

\maketitle

\section{Introduction}
In this paper we study factorizations of Dirac operators on singular fibrations of spin$^c$ manifolds into vertical and horizontal components. We restrict ourselves to {\em almost-regular fibrations} of spin$^c$ manifolds; these are defined to be Riemannian spin$^c$ manifolds $\overline M$ for which a proper Riemannian submersion to a spin$^c$ manifold $B$ is defined on the (open, dense) complement $M$ of a finite collection of compact embedded submanifolds which all have codimension strictly greater than 1. We will show that the Dirac operator $D_{\ov M}$ on the total space can be factorized as a tensor sum of a (vertically elliptic) family of Dirac operators $D_V$ with the horizontal Dirac operator $D_B$, up to an explicit curvature term $\Omega$. Thus, we obtain that (up to unitary equivalence and on a core of $D_{\overline{M}}$) we can write
\[
D_{\overline{M}} = D_V \otimes 1 + 1 \otimes_\nabla D_B + \Omega \, ,
\]
where $\nabla$ is a suitable metric connection on vertical spinors. 

We show that our factorization results fit well into the context of noncommutative geometry \cite{C94} by proving (Theorem \ref{thm:almost-regular-fact}) that the tensor sum represents the interior Kasparov product in bivariant K-theory (or KK-theory) \cite{Kas80b}. That is to say, we show that $D_V$ defines an unbounded Kasparov module (aka unbounded KK-cycle \cite{BJ83}) that represents a class in $KK(C_0(M),C_0(B))$ whereas $D_B$ defines a half-closed chain (in the sense of \cite{Hil10}), representing the fundamental class of $B$ in $KK(C_0(B),\C)$. Similarly, since we do not assume that the Riemannian manifold $\ov M$ is complete, the Dirac operator $D_{\ov M}$ defines a half-closed chain (and not necessarily a spectral triple) representing the fundamental class of $\ov M$ in $KK(C_0(\ov M),\C)$. The tensor sum factorization is then an unbounded representative of the interior Kasparov product %spectral triple \cite[Sect. IV.2.$\delta$]{C94} 
\begin{multline*}
\hot_{C_0(B)} : KK(C_0(M), C_0(B)) \times KK(C_0(B),\C) 
\to KK(C_0(M),\C)
\end{multline*}
in the sense that
\[
\imath^* [D_{\ov M}] = [D_V] \hot_{C_0(B)} [D_{B}] \, ,
\]
where $\imath^*$ is the pullback homomorphism in bivariant K-theory of the $*$-homomorphism $\imath : C_0(M) \to C_0(\ov M)$ given by extension by zero. We stress that the curvature term $\Omega$ is not visible at the level of bounded KK-theory. Indeed, it is the great advantage of working at the level of unbounded KK-cycles and half-closed chains that geometric information remains intact. In the course of the proof of the above result, we will establish (Proposition \ref{prop:diff-op-dense-subset}) that $\imath^*[D_{\ov M}] = [D_M]$, in terms of the Dirac operator $D_M$ on $M$. Then, the above factorization result follows from the result (Theorem \ref{thm:fact-KK}) that the Dirac operator $D_M$ represents the same class in $KK(C_0(M),\C)$ as the above tensor sum does. This is a generalization of our previous result \cite[Theorem 24]{KS16} to proper Riemannian submersions (not necessarily between compact manifolds), for which we make heavy use of a generalization to half-closed chains of a theorem by Kucerovsky \cite[Theorem 13]{Kuc97}, which we proved recently in \cite{KS17a}.

\bigskip

Let us spend a few words on the main motivation for this paper. This comes from a class of almost-regular fibrations derived from actions of tori on Riemannian manifolds (see Examples \ref{ex:action} and \ref{ex:action-spinc} below for more details). Consider a torus $G$ acting on a Riemannian spin$^c$ manifold $N$ (thus respecting the metric and spin$^c$ structure on $M$). Let $N_0 \subseteq N$ denote the principal stratum for the action and suppose that all other orbits are singular ({\it i.e.} there are no exceptional orbits). Then our main factorization result (Theorem \ref{thm:almost-regular-fact}) implies that the Dirac operator $D_N$ on $N$ can be written as 
\[
D_N = D_V \otimes 1 + 1 \otimes_\nabla D_{N_0/G} + \Omega \, ,
\]
in terms of a vertically elliptic family of Dirac operators $D_V$, the Dirac operator $D_{N_0/G}$ on the principal orbit space, and the curvature $\Omega$ of the proper Riemannian submersion $N_0 \to N_0/G$. Moreover, this tensor sum factorization is an unbounded representative of the interior Kasparov product:
\[
\imath^* [D_N] = [D_V] \otimes_{C_0(N_0/G)} [D_{N_0/G}] \, ,
\]
where $\imath^*$ is the pullback homomorphism in KK-theory of the embedding map $\imath : N_0 \into  N$. 

In particular, we are motivated by the implications the above factorization results will have for Dirac operators on toric noncommutative manifolds \cite{CL01,CD02,Sui14b}.

Interestingly, this very same class of examples also arises in the study of holonomy groupoids for singular foliations. Indeed, the orbits of an isometric Lie group action constitute an example of singular Riemannian foliations \cite{Mol88,Ale10} and, more generally, orbits of any Lie group action are an example of the almost-regular foliations that are analyzed using holonomy groupoids in \cite{AS09,Deb01}. We expect that ---at least in the context of these examples--- our factorization results of Dirac operators in unbounded KK-theory and the appearance of curvature can be applied to index theory on singular foliations as studied in \cite{C82,CS84,HS87,DL10,DS17}.

Another potential application of our results is to equivariant index theory. For instance, in \cite{HS17} a decomposition appears of a $G$-equivariant Dirac operator $D_M$ on a spin$^c$ manifold $M$ carrying a proper $G$-action. More precisely, after identifying $M \cong G \times_K N$ as a bundle over $G/K$ for a suitable subgroup $K \subseteq G$ and submanifold $N \subseteq M$, they write $D_M = D_{G/K} + D_N$ with $D_{G/K}$ a so-called $G$-differential operator and $D_N$ differentiating only in the vertical direction. It is an interesting question to see whether this formula can be cast as a tensor sum and represents the interior Kasparov product of the corresponding classes in bivariant K-theory.

On a more general level our paper could indicate how to proceed in the search for an unbounded, geometric version of bivariant K-theory. This is intimately related to the construction of a category of spectral triples, initiated in \cite{Mes09b} and ongoing in \cite{KL13,DGM15,FGMR16,Kaa16,ML18}. It is clear that having examples (or, even better, classes of examples) of unbounded representatives of the interior Kasparov product is of vital importance for finding the correct geometrical enrichment of the KK-category \cite{Cun83,Hig87,MN04}. A list of such examples includes \cite{BD13,BoeS10,BMS13,FR15b,KS17c,KS16}, and to which we now add the present general construction.

\subsection*{Acknowledgements}
We would like to thank Peter Hochs for a very useful suggestion on how to prove the compactness of the resolvent in the context of fiber bundles.
% for a highly stimulating remark concerning the ``locality'' of Kucerovsky's theorem.
%
%This work also benefited from various conversations with Magnus Goffeng and Bram Mesland.

We gratefully acknowledge the Syddansk Universitet Odense and the Radboud University Nijmegen for their financial support in facilitating this collaboration. 

During the initial stages of this research project the first author was supported by the Radboud excellence fellowship.
%This work was supported by the Radboud excellence fellowship.
%This research was initiated at the Radboud University during the   
%
%{\red Do you want to thank IMPAN, Bram?}
%

The first author was partially supported by the DFF-Research Project 2 ``Automorphisms and Invariants of Operator Algebras'', no. 7014-00145B and by the Villum Foundation (grant 7423).

The second author was partially supported by NWO under VIDI-grant 016.133.326.

\section{Unbounded Kasparov modules and fiber bundles}
\label{sect:fiberbundles}
Let $\pi : M \to B$ be a smooth fiber bundle with a {\it compact} model fiber $F$. We do not put any compactness restrictions on the manifolds $M$ and $B$ (but they are not allowed to have a boundary). By Ehresmann's fibration Theorem ({\it cf.} \cite[Lemma 17.2]{Mic08} for a proof) the above is equivalent to demanding that $\pi:M \to B$ is a proper and surjective submersion. We will assume that $M$ comes equipped with a Riemannian metric in the fiber direction. In particular, we have a fixed hermitian form
\[
\inn{\cd,\cd}_V : \sX_V(M) \ti \sX_V(M) \to C^\infty(M)
\]
on the $C^\infty(M)$-module of (complex) vertical vector fields $\sX_V(M) \su \sX(M)$. We will moreover assume that we have a smooth hermitian complex vector bundle $E \to M$. We denote the $C^\infty(M)$-module of smooth sections of $E$ by $\sE := \Ga^\infty(M,E)$ and the $C^\infty_c(M)$-module of compactly supported sections of $E$ by $\sE^c := \Ga^\infty_c(M,E)$. We denote the hermitian form by 
\[
\inn{\cd,\cd}_{\sE} : \sE \ti \sE \to C^\infty(M) \, . %\q \Tex{and} \q
%\inn{\cd,\cd}_{\sE^c} : \sE^c \ti \sE^c \to C^\infty_c(M)
\]
We remark that $\sE$ can be considered as a $C^\infty(M)$-$C^\infty(B)$-bimodule where the right action is given by $(s \cd f)(x) := s(x) \cd f(\pi(x))$ for all $s \in \sE$, $f \in C^\infty(B)$, $x \in M$. We finally fix a first-order differential operator
\[
\sD : \sE \to \sE
\]
which only differentiates in the fiber direction, or in other words: $\sD$ is $C^\infty(B)$-linear.

We let $\si_{\sD} : \Tex{Hom}_{C^\infty(M)}\big( \sX(M), C^\infty(M) \big) \to \Tex{End}_{C^\infty(M)}( \sE)$ denote the principal symbol of $\sD$ such that
\[
\si_{\sD}( df) = [\sD,f] \q f \in C^\infty(M) \, ,
\]
where $df \in \Tex{Hom}_{C^\infty(M)}\big( \sX(M), C^\infty(M) \big)$ denotes the exterior derivative of $f \in C^\infty(M)$. We let $d_V f \in \Tex{Hom}_{C^\infty(M)}\big( \sX_V(M), C^\infty(M) \big)$ denote the restriction of $df$ to the $C^\infty(M)$-submodule $\sX_V(M) \su \sX(M)$.

\begin{defn}\label{d:verell}
We say that $\sD$ is {\it vertically elliptic} when it holds for all $x \in M$ that the symbol $\si_{\sD}(df)(x) : E_x \to E_x$ is invertible whenever $f \in C^\infty(M)$ satisfies that $(d_V f)(x) : T_V(M)_x \to \CC$ is non-trivial, where $T_V(M) \to M$ denotes the complex vertical tangent bundle.
\end{defn}

The fiber-wise Riemannian metric $\inn{\cd,\cd}_V : \sX_V(M) \times \sX_V(M) \to C^\infty(M)$ gives rise to a Riemannian metric $\inn{\cd,\cd}_b : \sX(M_b) \ti \sX(M_b) \to C^\infty(M_b)$ on each of the fibers $M_b := \pi^{-1}(\{b\}) \su M$, $b \in B$. Indeed, the inclusion $i_b : M_b \to M$ induces an isomorphism $di_b : \sX(M_b) \to \sX_V(M) \ot_{C^\infty(M)} C^\infty(M_b)$ of $C^\infty(M_b)$-modules for all $b \in B$. In particular, we have an associated measure $\mu_b$ on $M_b$ for all $b \in B$. We may thus define the $C_c^\infty(B)$-linear map
\[
\rho : C^\infty_c(M) \to C_c^\infty(B) \q \rho(f)(b) := \int_{M_b} f|_{M_b} \, d \mu_b \q b \in B
\]
by integration over the fiber. This operation provides us with a $C^*$-correspondence $X$ from $C_0(M)$ to $C_0(B)$ defined as the completion of $\sE^c$ with respect to the inner product
\begin{equation}\label{eq:fibinn}
\inn{s,t}_X := \rho( \inn{s,t}_{\sE}) \q s,t \in \sE^c \, .
\end{equation}
Notice that the bimodule structure on $X$ is induced by the $C^\infty_c(M)$-$C^\infty_c(B)$-bimodule structure on $\sE^c$. We let $\mathbb{L}(X)$ denote the $C^*$-algebra of bounded adjointable operators on $X$ and denote the $*$-homomorphism providing the left action on $X$ by $m : C_0(M) \to \mathbb{L}(X)$.

We may promote our first-order differential operator $\sD : \sE \to \sE$ to an unbounded operator
\[
D_0 : \sE^c \to X \q D_0(s) := \sD(s) \, .
\]
We immediately record that the commutator
\[
[ D_0, m(f) ] : \sE^c \to X
\]
extends to a bounded operator $\de(f) : X \to X$ for all $f \in C_c^\infty(M)$. We emphasize that this is true even though the first-order differential operator $\sD : \sE \to \sE$ is not assumed to have finite propagation speed.

\begin{defn}\label{d:symm}
We say that $\sD : \sE \to \sE$ is \emph{symmetric} when $D_0 : \sE^c \to X$ is symmetric in the sense that
$\inn{D_0 (s), t}_X = \inn{s, D_0 (t)}_X$ for all $s,t \in \sE^c$.
\end{defn}

When $\sD : \sE \to \sE$ is symmetric we denote the closure of $D_0 : \sE^c \to X$ by $D := \ov{D_0} : \Tex{Dom}(D) \to X$. 

Our first main result can now be formulated. The full proof will occupy the remainder of this section.

\begin{thm}\label{thm:vertical-KK-cycle}
Suppose that the first-order differential operator $\sD : \sE \to \sE$ is vertically elliptic and symmetric. Then the triple $\big( C^\infty_c(M), X, D \big)$ is an odd unbounded Kasparov module from $C_0(M)$ to $C_0(B)$. Moreover, $\big( C^\infty_c(M),X,D \big)$ is even when $\sE$ comes equipped with a $\zz/2\zz$-grading operator $\ga \in \Tex{End}_{C^\infty(M)}(\sE)$ which anti-commutes with $\sD$ (in this case $\ga$ induces the grading operator on $X$).
\end{thm}
\begin{proof}
We already know that $m(f) : X \to X$ preserves the domain of $D$ and that the commutator $[D,m(f)] : \Tex{Dom}(D) \to X$ extends to a bounded operator on $X$ whenever $f \in C_c^\infty(M)$.

We need to show that $D : \Tex{Dom}(D) \to X$ is selfadjoint and regular and that $m(f) \cd (i + D)^{-1} : X \to X$ is a compact operator (in the sense of Hilbert $C^*$-modules) for all $f \in C_c^\infty(M)$. The selfadjointness and regularity is proved in Proposition \ref{p:regular} and the compactness result follows from Proposition \ref{p:localcompactI} and Proposition \ref{p:localcompactII}.  
\end{proof}

\subsection{Symmetry and regularity}\label{ss:symreg}
For each $b \in B$ we define the {\it localization of $X$} at the point $b \in B$ as the interior tensor product
\[
X_b := X \hot_{C_0(B)} \CC \, ,
\]
where the left action of $C_0(B)$ on $\CC$ is defined via the $*$-homomorphism $\Tex{ev}_b : C_0(B) \to \CC$ given by evaluation at the point $b \in B$. We remark that $X_b$ is a Hilbert space. 

For each $b \in B$, we denote the smooth sections of the smooth hermitian complex vector bundle $E|_{M_b} \to M_b$ by
\[
\sE_b := \Ga^\infty(M_b, E|_{M_b}) \cong \sE \ot_{C^\infty(M)} C^\infty(M_b) \, .
\]
We use the notation
\[
\inn{\cd,\cd}_{\sE_b} : \sE_b \ti \sE_b \to C^\infty(M_b)
\]
for the hermitian form (inherited from the hermitian form on $\sE$) and we let $L^2(\sE_b)$ denote the Hilbert space obtained as the completion of $\sE_b$ with respect to the inner product
\[
\inn{s,t} := \int_{M_b} \inn{s,t}_{\sE_b} \, d\mu_b \, .
\]

We remark that the map $\sE^c \to \sE_b$ defined by $s \mapsto s|_{M_b}$, or alternatively by $s \mapsto s \ot 1$, is a surjection. The next lemma can now be verified by computing the inner products involved:

\begin{lma}\label{l:uniloc}
Let $b \in B$. The linear map 
\[
\sE^c \ot_{C^\infty_c(B)} \CC \to \sE_b \q s \ot \la \mapsto s|_{M_b} \cd \la %\q  s \in \sE^c \, , \, \, \la \in \CC
\]
induces a unitary isomorphism of Hilbert spaces: $X_b \cong L^2(\sE_b)$.
\end{lma}

We recall that the model fiber $F$ of our smooth fiber bundle $\pi : M \to B$ is assumed to be compact. This plays an important role in the following proposition.

\begin{prop}\label{p:regular}
Suppose that $D_0 : \sE^c \to X$ is symmetric. Then the closure $D := \ov{D_0} : \Tex{Dom}(D) \to X$ is selfadjoint and regular.
\end{prop}
\begin{proof}
By \cite[Theorem 1.18]{Pie06} and \cite[Theorem 5.8]{KL12} it suffices to verify that the induced unbounded operator (the localization)
\[
D \hot 1 : \Tex{Dom}(D \hot 1) \to X \hot_{C_0(B)} \CC
\]
is selfadjoint for all points $b \in B$ (recall that any pure state on $C_0(B)$ is of the form $\Tex{ev}_b : C_0(B) \to \CC$ for some $b \in B$). Moreover, we recall that the image of the map
\[
\sE^c \to X_b \q s \mapsto s \hot 1
\]
is a core for $D \hot 1$ for all $b \in B$. We denote this core by $\sE^c \ot_{C^\infty_c(B)} \CC \subseteq X_b$. 

Let thus $b \in B$ be fixed. Under the identification $X_b \cong L^2(\sE_b)$ (of Lemma \ref{l:uniloc}) we have that the core $\sE^c \ot_{C^\infty_c(B)} \CC$ corresponds to $\sE_b \su L^2(\sE_b)$ and that $D_0 \ot 1 : \sE^c \ot_{C^\infty_c(B)} \CC \to X_b$ corresponds to a first-order symmetric differential operator $(D_0)_b : \sE_b \to L^2(\sE_b)$. Since $M_b$ is {\it compact} we have that the closure $D_b := \ov{(D_0)_b}$ is selfadjoint, see \cite[Corollary 10.2.6]{HR00}. Moreover, we obtain that $D \hot 1$ agrees with $D_b$ under the unitary isomorphism of Lemma \ref{l:uniloc}. But this means that the localization $D \hot 1$ is selfadjoint for all $b \in B$ and the proposition is therefore proved.
\end{proof}

\subsection{Restriction to an open subset}
We continue in this subsection under the assumption that $\sD : \sE \to \sE$ is symmetric. We thus know from Proposition \ref{p:regular} that the closure of $D_0$, $D := \ov{D_0} : \Tex{Dom}(D) \to X$ is a selfadjoint and regular unbounded operator. We are now going to provide a local criteria that will later on allow us to verify the local compactness of the resolvent of $D$ (under an extra ellipticity condition). %\cong \sE \ot_{C^\infty(M)} C^\infty(U)$ 

\begin{rem}
One might of course believe that the local compactness of the resolvent of $D : \Tex{Dom}(D) \to X$ would follow immediately from the compactness of the resolvents of all the localized operators $D_b : \Tex{Dom}(D_b) \to L^2(\sE_b)$, $b \in B$. This kind of argumentation is however erroneous. Indeed, consider the simple case where both the fiber bundle $\pi : M \to B$ and the vector bundle $E \to M$ are trivial. In this case we may identify $X$ with the standard module $X \cong C_0(B) \hot L^2( F )^{\op k}$ and the compactness of the resolvents of all the localized operators $D_b$ amounts to the {\em pointwise} compactness of $(i + D)^{-1}$. However, the compact operators on $X$ are given by the {\em operator norm continuous} maps $B \to \mathbb{K}( L^2(F)^{\op k})$ which vanish at infinity, see \cite[Lemma 4]{Kas80a}. The selfadjointness and regularity of $D$ only implies that the resolvent $(i + D)^{-1} : B \to \mathbb{L}( L^2(F)^{\op k})$ is continuous with respect to the {\em $*$-strong topology}. 

In general, the task of proving the local compactness of the resolvent of $D$ is also made more complicated by the fact that our vector bundle $E \to M$ need not be trivial over open subsets of the form $\pi^{-1}(V) \subseteq M$.
%
%I HAVE TO DOUBLE CHECK THAT THIS ``ERRONEOUS'' ARGUMENT DOES NOT WORK.
%
%{\blu Internal: The map $K : \{0\} \cup \{1/n \mid n \in \nn\} \to \mathbb{K}( \ell^2(\nn))$ given by $K(1/n) := e_{n,n}$ and $K(0) = 0$ is continuous in the $*$-strong topology but not in operator norm, even though $K(1/n)$ has rank one for all $n \in \nn$. Notice also that the space $\{0\} \cup \{1/n \mid n \in \nn\}$ is compact and Hausdorff.}
\end{rem}

%resolvents of the localized over $C_0$$$ w
%We emphasize that the working of the above argument relies in a crucial way on the {\it operator norm continuity} of the resolvent $R : \pi(U) \to \mathbb{L}\big( L^2(\mathbb{B}_\de(0))^{\op k}\big)$. It could in principle have happened that $R$ was only continuous with respect to the $*$-strong topology and in this case we would not have been able to conclude anything about the compactness of operators of the form $m_U(\chi) \cd ( 1 + D_U^* D_U)^{-1} : X_U \to X_U$. Indeed, the compact operators on the Hilbert $C^*$-module
%\[
%X_U \cong C_0(\pi(U)) \hot L^2\big( \mathbb{B}_\de(0)\big)^{\op k}
%\]
%agree with all \emph{operator norm continuous maps} $\pi(U) \to \mathbb{K}\big( L^2\big( \mathbb{B}_\de(0)\big)^{\op k} \big)$ that vanish at infinity. 
%
%\end{rem}

For any open subset $U \su M$ we let $\sE_U := \Ga^\infty(U,E|_U)$ denote the smooth sections of the restriction of $E \to M$ to $U$. The notation $\sE_U^c := \Ga^\infty_c(U,E|_U)$ refers to the smooth compactly supported sections. We define the $C^*$-correspondence $X_U$ from $C_0(U)$ to $C_0(\pi(U))$ as the completion of $\sE_U^c$ with respect to (the norm coming from) the inner product
\[
\inn{s,t}_{X_U} := \inn{i_U(s),i_U(t)}_X \in C_0\big( \pi(U) \big) \q s,t \in \sE_U^c \, ,
\]
where $i_U : \sE_U^c \to \sE^c$ denotes the inclusion given by extension by zero. We remark that the bimodule structure on $X_U$ is induced by the $C_c^\infty(U)$-$C_c^\infty(\pi(U))$-bimodule structure on $\sE_U^c$. Our first-order differential operator $\sD : \sE \to \sE$ then restricts to a first-order differential operator $\sD_U : \sE_U \to \sE_U$ which we may promote to a {\it symmetric} unbounded operator 
\[
(D_U)_0 : \sE^c_U \to X_U \, .
\]
We emphasize that the closure $D_U := \ov{ (D_U)_0 } : \Tex{Dom}(D_U) \to X_U$ need {\it not} be selfadjoint. 

We equip $\Tex{Dom}(D_U)$ with the structure of a Hilbert $C^*$-module over $C_0(\pi(U))$ by defining the inner product
\[
\inn{s,t}_{D_U} := \inn{s,t}_{X_U} + \inn{D_U(s), D_U(t)}_{X_U} \, ,
\]
and the right action induced by the right action of $C_0(\pi(U))$ on $X_U$. The inclusion $i_{D_U} : \Tex{Dom}(D_U) \to X_U$ is then a bounded operator between the two Hilbert $C^*$-modules. Remark that the existence of an adjoint to this inclusion is equivalent to the {\it regularity} of $D_U$ so for the moment we only know that the inclusion is bounded.

We denote the left actions by
\[
m : C_0(M) \to \mathbb{L}(X) \q \Tex{and} \q m_U : C_0(U) \to \mathbb{L}(X_U) \, .
\]

Our localization result for resolvents can now be stated and proved:

\begin{prop}\label{p:localcompactI}
Suppose that, for each $p \in M$ we may find an open subset $U \su M$ with $p \in U$ such that the bounded operator
\[
\begin{CD}
\Tex{Dom}(D_U) @>{i_{D_U}}>> X_U @>{m_U(\chi)}>> X_U
\end{CD}
\]
is compact for all $\chi \in C_c^\infty(U)$. Then we have that
\[
m(f) \cd (i + D)^{-1} : X \to X
\]
is compact for all $f \in C_c^\infty(M)$.
\end{prop}
\begin{proof}
Let us fix an $f \in C_c^\infty(M)$. Without loss of generality we may assume that $\Tex{supp}(f) \su U$ for some open subset $U \su M$ satisfying the assumptions to this proposition. Moreover, we may assume that $f \geq 0$ and that $f$ has a smooth square root.

The proof runs in four steps.

\begin{enumerate}
\item Consider the interior tensor product $X_U \hot_{C_0(\pi(U))} C_0(B)$ where $C_0(\pi(U))$ acts on $C_0(B)$ via the $*$-homomorphism $C_0(\pi(U)) \to C_0(B)$ given by extension by zero. We start by showing that the composition
\[
\begin{CD}
\Tex{Dom}(D_U \hot 1) @>{i_{D_U \hot 1}}>> X_U \hot_{C_0(\pi(U))} C_0(B) @>{m_U(\sqrt{f}) \hot 1}>> X_U \hot_{C_0(\pi(U))} C_0(B)
\end{CD}
\]
is a compact operator. To this end, we remark that $\Tex{Dom}(D_U \hot 1)$ is unitarily isomorphic to $\Tex{Dom}(D_U) \hot_{C_0(\pi(U))} C_0(B)$ and that the map $i_{D_U \hot 1} : \Tex{Dom}(D_U \hot 1) \to X_U \hot_{C_0(\pi(U))} C_0(B)$ corresponds to $i_{D_U} \hot 1 : \Tex{Dom}(D_U) \hot_{C_0(\pi(U))} C_0(B) \to X_U \hot_{C_0(\pi(U))} C_0(B)$ under this unitary isomorphism. It thus suffices to show that the bounded operator
\[
\big( m_U(\sqrt{f}) \cd i_{D_U} \big) \hot 1 : \Tex{Dom}(D_U) \hot_{C_0(\pi(U))} C_0(B) \to X_U \hot_{C_0(\pi(U))} C_0(B)
\]
is compact. But this follows from \cite[Proposition 4.7]{Lan95} since $m_U(\sqrt{f}) \cd i_{D_U} : \Tex{Dom}(D_U) \to X_U$ is compact by assumption and since $C_0(\pi(U))$ acts on $C_0(B)$ by compact operators.
\item The map $\sE_U^c \ot_{C_c^\infty(\pi(U))} C_c^\infty(B) \to \sE^c$ defined by $s \ot g \mapsto i_U(s) \cd (g \ci \pi)$ induces a $C_0(B)$-linear isometry $j : X_U \hot_{C_0(\pi(U))} C_0(B) \to X$ and we may thus conclude (using $(1)$) that the composition
\[
\begin{CD}
\Tex{Dom}(D_U \hot 1) @>{( m_U(\sqrt{f}) \hot 1) \cd i_{D_U \hot 1}}>> X_U \hot_{C_0(\pi(U))} C_0(B) 
@>{j}>> X
\end{CD}
\]
is a compact operator. Remark that we do not need 
\[
j : X_U \hot_{C_0(\pi(U))} C_0(B) \to X
\]
to be adjointable since the compact operators are compatible with {\it left multiplication} by bounded operators that are linear over the base (the corresponding statement for {\it right multiplication} need not be true).
\item The multiplication operator $r_U(\sqrt{f}) : \sE^c \to \sE^c_U \ot_{C_c^\infty(\pi(U))} C_c^\infty(B)$ induces a bounded {\it adjointable} operator
\[
r_U(\sqrt{f}) : \Tex{Dom}(D) \to \Tex{Dom}(D_U \hot 1) \, .
\]

To prove this claim, we recall that $\de( \sqrt{f}) : X \to X$ denotes the bounded adjointable extension of the commutator $[D_0, m(\sqrt{f})] : \sE^c \to X$. 

The fact that $r_U(\sqrt{f})$ is bounded then follows since
\[
\begin{split}
\| r_U(\sqrt{f})(s) \|_{\Tex{Dom}(D_U \hot 1)} & = \| (i + D) (\sqrt{f} \cd s) \|_X \\
& \leq \| m(\sqrt{f}) \|_\infty \cd \| s \|_{\Tex{Dom}(D)} + \| \de( \sqrt{f}) \|_\infty \cd \| s \|_X
\end{split}
\]
for all $s \in \sE^c$.

It can then be verified by a direct computation that the adjoint of $r_U(\sqrt{f})$ is given by the expression:
\[
\begin{split}
r_U(\sqrt{f})^* 
& = m(\sqrt{f}) \cd j  - (i + D)^{-1} \cd \de(\sqrt{f}) \cd j  \\
& \qqq - (1 + D^2)^{-1} \cd \de(\sqrt{f}) \cd (i + D) \cd j
\end{split}
\]
on the dense subspace of $\Tex{Dom}(D_U \hot 1)$ provided by the tensor product $\sE^c_U \ot_{C_c^\infty(\pi(U))} C_c^\infty(B)$.

This proves the claim.
\item We end the proof of the proposition by concluding that the composition
\[
\begin{CD}
\Tex{Dom}(D) @>{i_D}>> X @>{m(f)}>> X
\end{CD}
\]
is a compact operator. Indeed, this composition can be rewritten as the composition
\[
\begin{CD}
\Tex{Dom}(D) @>{r_U(\sqrt{f})}>> \Tex{Dom}(D_U \hot 1) @>{i_{D_U \hot 1}}>> X_U \hot_{C_0(\pi(U))} C_0(B) \\ 
@>{m_U(\sqrt{f}) \hot 1}>> X_U \hot_{C_0(\pi(U))} C_0(B) @>{j}>> X \qqq
\end{CD}
\]
but this latter composition is compact by a combination of $(2)$ and $(3)$.
\end{enumerate}

\end{proof}

\subsection{Compactness of local resolvents}
Throughout this subsection we will suppose that the first-order differential operator $\sD : \sE \to \sE$ is vertically elliptic and symmetric.

Let us fix an open subset $U \su M$ such that the following holds:

\begin{ass}\label{a:localI}
\begin{enumerate}
\item There exists a diffeomorphism $\psi : U \to \pi(U) \ti \mathbb{B}_\de(0)$ such that $p_1 \ci \psi = \pi$ where $p_1 : \pi(U) \ti \mathbb{B}_\de(0) \to \pi(U)$ denotes the projection onto the first component and where $\mathbb{B}_\de(0) \su \rr^{\dim (F)}$ denotes the open ball with center $0$ and radius $\de > 0$.
\item There exists a unitary smooth trivialization $\phi: E|_U \to U \ti \CC^k$ of the restriction of the vector bundle $E \to M$ to $U$.
\end{enumerate}
\end{ass}

We define the vertical coordinates associated to $\psi$ by $y_i := r_i \ci p_2 \ci \psi : U \to \rr$, $i = 1,\ldots, \dim (F)$, where $r_i : \mathbb{B}_\de(0) \to \rr$ are restrictions of the standard coordinates on $\rr^{\dim (F)}$ and $p_2 : \pi(U) \ti \mathbb{B}_\de(0) \to \mathbb{B}_\de(0)$ is the projection onto the second component. We then have the associated smooth map
\[
g^V : U \to \Tex{GL}_+( \rr^{\dim (F)}) \q g^V_{ij} := \inn{\pa/\pa y_i, \pa/\pa y_j}_V \, .
\]

The next lemma follows by a straightforward computation of inner products.

\begin{lma}
The $C^\infty_c(\pi(U))$-module isomorphism
\[
\al : \sE^c_U \to C_c^\infty\big( \pi(U) \ti \mathbb{B}_\de(0) \big)^{\op k}
\]
defined by $\al(s) := \big( (\phi \ci s) \cd \Tex{det}(g^V)^{1/4} \big) \ci \psi^{-1}$ induces a unitary isomorphism of Hilbert $C^*$-modules over $C_0( \pi(U))$:
\[
\al : X_U \to C_0\big( \pi(U) \big) \hot L^2\big( \mathbb{B}_\de(0) \big)^{\op k} \, .
\]
\end{lma}

We are going to study the closed and symmetric unbounded operator 
\[
\begin{split}
D^\al := \al \ci D_U \ci \al^{-1} & : \Tex{Dom}(D^\al) = \al\big( \Tex{Dom}(D_U) \big) \\ 
& \q \to
C_0\big( \pi(U) \big) \hot L^2\big( \mathbb{B}_\de(0) \big)^{\op k}  \, .
\end{split}
\]
We immediately remark that $C_c^\infty( \pi(U) \ti \mathbb{B}_\de(0))^{\op k}$ is a core for $D^\al$. Moreover, we can find smooth maps $A_j, B : \pi(U) \ti \mathbb{B}_\de(0) \to M_k(\CC)$, $j = 1,\ldots,\dim (F)$, such that
\[
\begin{split}
D^\al_0 = \sum_{j = 1}^{\dim ( F )} A_j \cd \frac{\pa}{\pa r_j} + B & : C_c^\infty( \pi(U) \ti \mathbb{B}_\de(0))^{\op k} \\
& \q \to C_0\big( \pi(U) \big) \hot L^2\big( \mathbb{B}_\de(0) \big)^{\op k}  \, ,
\end{split}
\]
where $D^\al_0$ denotes the restriction of $D^\al$ to the core $C_c^\infty( \pi(U) \ti \mathbb{B}_\de(0))^{\op k}$. For each $x \in \pi(U)$ we define the first-order differential operator
\[
\sD^\al_x := \sum_{j = 1}^{\dim ( F )} A_j(x,\cd) \cd \frac{\pa}{\pa r_j} + B(x,\cd)
: C^\infty( \mathbb{B}_\de(0))^{\op k} \to C^\infty( \mathbb{B}_\de(0))^{\op k} \, .
\]
The vertical ellipticity of $\sD : \sE \to \sE$ then implies that $\sD^\al_x$ is elliptic for all $x \in \pi(U)$. For each $x \in \pi(U)$, we let
\[
D^\al_x : \Tex{Dom}(D^\al_x) \to L^2( \mathbb{B}_\de(0))^{\op k}
\]
denote the closure of the unbounded operator $(D^\al_x)_0 : C^\infty_c( \mathbb{B}_\de(0))^{\op k} \to L^2( \mathbb{B}_\de(0))^{\op k}$ induced by the first-order differential operator $\sD^\al_x$.

By passing to a smaller open subset $U' \su U$ if necessary we may assume that the following holds:

\begin{ass}\label{a:localII}
\begin{enumerate}
\item For each $x \in \pi(U)$, the smooth maps $A_j(x,\cd)$, $B(x,\cd) : \mathbb{B}_\de(0) \to M_k(\CC)$, $j = 1,\ldots, \dim(F)$, are bounded. % and all derivatives with respect to the vertical coordinates $r_i \ci p_2 : \pi(U) \ti \mathbb{B}_\de(0) \to \rr$, $i = 1,\ldots, \dim(F)$ are bounded as well.
\item The maps $A_j, B : \pi(U) \to C_b( \mathbb{B}_\de(0), M_k(\CC))$, $j = 1,\ldots, \dim(F)$, are continuous, where $C_b(\mathbb{B}_\de(0), M_k(\CC))$ is equipped with the supremum norm. %(where the smooth and bounded functions on $\mathbb{B}_\de(0)$ are equipped with the usual Fr\'echet-space topology). Thus, all the vertical derivatives
%\[
%\frac{\pa^{|I|} A_j}{\pa (r \ci p_2)_I}, \frac{\pa^{|I|} B}{\pa (r \ci p_2)_I} 
%: \pi(U) \to C_b( \mathbb{B}_\de(0), M_k(\CC))
%\]
%where $I \in \nn_0^{\dim(F)}$, are continuous in supremum-norm.
\item $D^\al$ satisfies the following uniform G\aa rding's inequality: There exists a constant $C > 0$ such that
\[
\| \xi \|_2 + \| \sD^\al_x(\xi) \|_2  
\geq C \cd \big( \| \xi \|_2 + \sum_{j = 1}^{\dim(F)}\| \pa \xi/\pa r_j \|_2 \big) \, ,  
\]
for all $\xi \in C_c^\infty(\mathbb{B}_\de(0))^{\op k}$ and all $x \in \pi(U)$, where $\| \cd \|_2$ denotes the $L^2$-norm on $L^2\big( \mathbb{B}_\de(0) \big)^{\op k}$. Indeed, this is a consequence of the proof of the usual G\aa rding's inequality, see for example \cite[Theorem 10.4.4]{HR00}.
\end{enumerate}
\end{ass}

Let us apply the notation
\[
\Na : \Tex{Dom}(\Na) \to L^2\big( \mathbb{B}_\de(0) \big)^{\op \dim(F) \cd k}
\]
for the closure of the gradient
\[
\Na_0 : C_c^\infty(\mathbb{B}_\de(0))^{\op k} \to L^2\big( \mathbb{B}_\de(0) \big)^{\op \dim(F) \cd k} \, \, \,
\Na_0( \xi) = \ma{c}{ \pa \xi/ \pa r_1 \\ \vdots \\ \pa \xi/ \pa r_{\dim(F)}} \, .
\]

The above assumptions imply the following:

\begin{lma}\label{l:domain}
We have that $\Tex{Dom}(\Na) = \Tex{Dom}( D^\al_x )$ and there exist constants $C_0, C_1 > 0$ such that
\[
C_1 \cd \| \xi \|_\Na \geq \| \xi \|_{D^\al_x} \geq C_0 \cd \| \xi \|_\Na \, ,
\]
for all $x \in \pi(U)$.  Moreover, we have that
\[
D^\al_x(\xi) = A(x,\cd) \cd \Na(\xi) + B(x,\cd)(\xi) \, ,
\]
for all $\xi \in \Tex{Dom}(\Na)$ and all $x \in \pi(U)$, where
\[
\begin{split}
A(x,\cd) & := (A_1(x,\cd),\ldots, A_{\dim(F)}(x,\cd)) \\ 
& \qq : L^2\big( \mathbb{B}_\de(0) \big)^{\op \dim(F) \cd k} \to L^2 \big( \mathbb{B}_\de(0) \big)^{\op k} \, .
\end{split}
\]
\end{lma}

Define the map
\[
R : \pi(U) \to \mathbb{L}\big( L^2(\mathbb{B}_\de(0))^{\op k} \big) \q R(x) := \big( 1 + (D^\al_x)^* D^\al_x \big)^{-1}
\]
as well as its square root $R^{1/2} : x \mapsto R(x)^{1/2}$, $x \in \pi(U)$. We remark that it follows by Lemma \ref{l:domain} that $\Tex{Im}( R^{1/2}(x)) = \Tex{Dom}(\Na)$ for all $x \in \pi(U)$ and that the map
\[
\begin{split}
& \Na \cd R^{1/2} : \pi(U) \to \mathbb{L}\big( L^2(\mathbb{B}_\de(0))^{\op k},  L^2(\mathbb{B}_\de(0))^{\op \dim(F) \cd k}\big) \\
& \Na \cd R^{1/2}(x) := \Na \cd  \big( 1 + (D^\al_x)^* D^\al_x \big)^{-1/2}
\end{split}
\]
is well-defined. Moreover, Lemma \ref{l:domain} implies that
\[
\sup_{x \in \pi(U)} \| \Na \cd R^{1/2}(x) \|_\infty < \infty \, ,
\]
where $\| \cd \|_\infty$ refers to the operator norm. %since
%\[
%\| \Na \cd R^{1/2}(z) \ze \|_2  \leq \| R^{1/2}(z) \ze \|_\Na \leq \frac{1}{C_0} \| R^{1/2}(z) \ze \|_{D^\al_z} \leq \frac{1}{C_0} \| \ze \|_2
%\]
%for all $\ze \in L^2( \mathbb{B}_\de(0))^{\op k}$ and all $z \in \pi(U)$.
%%\[
%\sup
%\]

\begin{lma}\label{l:opnorcon}
The maps $R : \pi(U) \to \mathbb{L}\big( L^2(\mathbb{B}_\de(0))^{\op k} \big)$ and $\Na \cd R : \pi(U) \to \mathbb{L}\big( L^2(\mathbb{B}_\de(0))^{\op k},  L^2(\mathbb{B}_\de(0))^{\op \dim(F) \cd k}\big)$ are continuous in operator norm.
\end{lma}
\begin{proof}
We will only prove the statement for the map $\Na \cd R$ since the argument in the case of $R$ is similar but easier.

Thus, let $\xi \in L^2(\mathbb{B}_\de(0))^{\op k}$ and $\eta \in L^2(\mathbb{B}_\de(0))^{\op \dim(F) \cd k}$ as well as $x,y \in \pi(U)$ be given. To ease the notation, we define the bounded operator 
\[
T(y) := R^{1/2}(y) \big( \Na \cd R^{1/2}(y) \big)^* : L^2(\mathbb{B}_\de(0))^{\op \dim(F) \cd k} \to L^2(\mathbb{B}_\de(0))^{\op k}
\]
and remark that 
\[
T(y)^* = \Na \cd R(y) : L^2(\mathbb{B}_\de(0))^{\op k} \to L^2(\mathbb{B}_\de(0))^{\op \dim(F) \cd k} \, .
\]
Using the resolvent identity we have that
\[
\begin{split}
\Na \cd ( R(x) - R(y)) 
& = \Na \cd R^{1/2}(y) ( D_y^\al R^{1/2}(y) )^* D_y^\al R(x) \\ 
& \q + \Na R(y) R(x) - \Na R(y) \\
& = ( D_y^\al T(y) )^* D_y^\al R(x) + T(y)^* R(x) - T(y)^* \\
& = ( D_y^\al T(y) )^* D_y^\al R(x) - T(y)^* (D^\al_x)^*D^\al_x R(x) 
\end{split}
\]
and thus that
\[
\begin{split}
& \inn{ \Na \cd  ( R(x) - R(y) ) \xi , \eta}  \\
& \q = \binn{ D^\al_y R(x) \xi, D^\al_y T(y) \eta } - \binn{ (D^\al_x)^*D^\al_x R(x) \xi, T(y) \eta } \\
& \q = \binn{ \big( A(y,\cd) - A(x,\cd) \big) \Na \cd R(x) \xi, D^\al_y T(y) \eta} \\
& \qq + \binn{ \big( B(y,\cd) - B(x,\cd) \big) R(x) \xi, D^\al_y T(y) \eta} \\
& \qqq + \binn{D^\al_x R(x) \xi, \big( A(y,\cd) - A(x,\cd) \big) \Na \cd T(y) \eta} \\
& \qqqq + \binn{D^\al_x R(x) \xi, \big( B(y,\cd) - B(x,\cd) \big) T(y) \eta} \, .
\end{split}
\]
But this implies that
\[
\begin{split}
\| \Na \cd  ( R(x) - R(y) ) \|_\infty
& \leq 2 \cd \| A(y,\cd) - A(x,\cd) \|_\infty \cd K^2  \\ 
& \qq +  2 \cd \| B(y,\cd) - B(x,\cd) \|_\infty \cd K  \, ,
\end{split}
\]
where $K := \sup_{z \in \pi(U)} \| \Na \cd R^{1/2}(z) \|_\infty$. This estimate proves the lemma.
\end{proof}

\begin{prop}\label{p:localcompactII}
Suppose that $\sD : \sE \to \sE$ is vertically elliptic and symmetric. For each $p \in M$, there exists an open subset $U \subseteq M$ with $p \in U$ such that
\begin{enumerate}
\item The closed and symmetric unbounded operator $D_U : \Tex{Dom}(D_U) \to X_U$ is regular.
\item The bounded adjointable operator
\[
m_U(\chi) \cd ( 1 + D_U^* D_U)^{-1} : X_U \to X_U
\]
is compact for all $\chi \in C_c^\infty( U )$.
\end{enumerate}
In particular, it holds that the composition
\[
\begin{CD}
\Tex{Dom}(D_U) @>{i_{D_U}}>> X_U @>{m_U(\chi)}>> X_U
\end{CD}
\]
is a compact operator for all $\chi \in C_c^\infty( U)$.
\end{prop}
\begin{proof}
We choose the open subset $U \subseteq M$ with $p \in U$ such that Assumption \ref{a:localI} and Assumption \ref{a:localII} are satisfied.

Let us first show that $D^\al$ is regular. Since $D_U$ and $D^\al$ are unitarily equivalent this is equivalent to the regularity of $D_U$. We show that the operator norm continuous map $R : \pi(U) \to \mathbb{L}\big( L^2(\mathbb{B}_\de(0))^{\op k} \big)$ satisfies that $\Tex{Im}(R) \su \Tex{Dom}\big( (D^\al)^* D^\al \big)$ and that $\big( 1 + (D^\al)^* D^\al \big) R(\xi) = \xi$ for all $\xi \in C_0\big( \pi(U) \big) \hot L^2( \mathbb{B}_\de(0))^{\op k}$. Remark here that $R$ is identified with a bounded adjointable operator on $C_0\big( \pi(U) \big) \hot L^2( \mathbb{B}_\de(0))^{\op k}$ in the obvious way. Now, by Lemma \ref{l:domain} and Lemma \ref{l:opnorcon} we have that $\Tex{Im}(R) \su \Tex{Dom}(D^\al)$. For each $\xi \in C_0( \pi(U)) \hot L^2( \mathbb{B}_\de(0))^{\op k}$ and each $\eta \in \Tex{Dom}(D^\al)$ we then have that
\[
\inn{ D^\al R (\xi), D^\al(\eta)}(x) = \inn{D^\al_x R(x)( \xi(x)), D^\al_x( \eta(x)) }
= \inn{ (1 - R)(\xi), \eta}(x)
\]
for all $x \in \pi(U)$. But this shows that $\Tex{Im}( D^\al R ) \su \Tex{Dom}\big( ( D^\al)^* \big)$ and hence that $\Tex{Im}(R) \su \Tex{Dom}\big( (D^\al)^* D^\al \big)$. Moreover, we may conclude that $( D^\al)^* D^\al R = 1 - R$. We have thus proved that $D^\al$ is regular with resolvent $(1 + (D^\al)^* D^\al)^{-1} = R$.

To show that $m_U(\chi) \cd ( 1 + D_U^* D_U)^{-1} : X_U \to X_U$ is a compact operator for all $\chi \in C_c^\infty(U)$ it now suffices to show that $m(f) \cd R(x) : L^2\big( \mathbb{B}_\de(0)\big)^{\op k} \to L^2\big( \mathbb{B}_\de(0)\big)^{\op k}$ is a compact operator for all $x \in \pi(U)$ and all $f \in C_c^\infty( \mathbb{B}_\de(0))$. But this follows from Rellich's Lemma since $\sD^\al_x : C^\infty\big( \mathbb{B}_\de(0)\big)^{\op k} \to C^\infty\big( \mathbb{B}_\de(0)\big)^{\op k}$ is an elliptic first-order differential operator, see \cite[Lemma 10.4.3]{HR00}. 
\end{proof}

\section{Factorization in unbounded KK-theory}
We remain in the setting described in Section \ref{sect:fiberbundles}, thus we consider a smooth fiber bundle $\pi : M \to B$ with a compact model fiber $F$. We will assume that both $M$ and $B$ are Riemannian and that the submersion $\pi : M \to B$ is \emph{Riemannian} thus that the derivative $d \pi(x) : (T_V M)_x^\perp \to (T B)_{\pi(x)}$ is an isometry for all $x \in M$ ($T_V M \to M$ denotes the smooth hermitian vector bundle of vertical tangent vectors). We will finally assume that $M$ and $B$ are both \emph{spin$^c$ manifolds}.

The second main result of this paper is the factorization in unbounded KK-theory of the Dirac operator $D_M$ on the total manifold $M$ in terms of a vertical operator $D_V$ and the Dirac operator $D_B$ on the base manifold $B$. This factorization result holds up to an explicit curvature term, which is invisible at the level of bounded KK-theory.

We work in the case where $M$ and $B$ are both even dimensional, but notice that our results can be readily translated to the remaining three cases (counting parity of dimensions).

The Riemannian metrics on $M$ and $B$ will be denoted by $\inn{\cd,\cd}_M : \sX(M) \ti \sX(M) \to C^\infty(M)$ and $\inn{\cd,\cd}_B : \sX(B) \ti \sX(B) \to C^\infty(B)$, respectively.

%{\red We could state our first main theorem here?}

%In this section we come to the main result of this paper, which is the factorization in unbounded KK-theory of the Dirac operator $D_M$ on $M$ in terms of a vertical operator $S$ and the Dirac operator $D_B$ on a base manifold $B$ for so-called almost-regular fibrations. Let us start with a precise definition.

%% {\blu Elaborate on  assumptions on codimension of $\pi^{-1} (B\setminus B)$, which should yield a family of $\mu_n$ compactly supported in $B$, converging to 1 pointswise, and such that 
%% $$
%% \|\pi^* (d \mu_n)\|_{L^2(M)} 
%% $$
%% is uniformly bounded. (but see Proposition \ref{prop:ess-self-adj} below).
%% }

\subsection{The vertical unbounded Kasparov module}
In order to link the Dirac operators $D_B$ and $D_M$ via a tensor-sum factorization, we start by constructing a vertical Dirac operator $\sD_V : \sE_V \to \sE_V$. This vertical Dirac operator will be an odd symmetric, vertically elliptic, first-order differential operator acting on the smooth sections $\sE_V$ of a $\zz/2\zz$-graded smooth hermitian vector bundle $E_V \to M$ (see Definition \ref{d:symm} and Definition \ref{d:verell}). In particular, using Theorem \ref{thm:vertical-KK-cycle}, we will obtain an even unbounded Kasparov module from $C_0(M)$ to $C_0(B)$.%, see Proposition \ref{p:unbkas}.

%% \section{Riemannian submersions of spin$^c$ manifolds}
%% We now apply the results of the previous section to the class of proper submersions where the manifolds $M$ and $B$ are Riemannian spin$^c$ manifolds. %We start by briefly recalling the notion of a Riemannian submersion where 
%% Note that we do not impose any compactness conditions on the manifolds $M$ and $B$. However, we do assume that they have no boundary. 
We follow the approach and notation of \cite{KS16}, where we studied ---actually, as a preparation for the present paper--- Riemannian submersions of {\em compact} Riemannian spin$^c$ manifolds. %The results obtained therein can be extended to locally compact manifolds, essentially because all definitions and results on connections are local. 

Working with the unital function algebras $C^\infty(M)$ and $C^\infty(B)$, we consider the smooth sections $\sE_M := \Ga^\infty(M,E_M)$ and $\sE_B := \Ga^\infty(B,E_B)$ of the $\zz/2\zz$-graded spinor bundles $E_M \to M$ and $E_B \to B$, respectively. The spin$^c$-structures provide us with even isomorphisms
\[
c_M : \Cl(M) \to \Tex{End}_{C^\infty(M)}(\sE_M),\qquad 
c_B : \Cl(B) \to \Tex{End}_{C^\infty(B)}(\sE_B) \, .
\]
We fix even hermitian Clifford connections $\Na^{\sE_M} : \Ga^\infty(M,E_M) \to \Ga^\infty(M,E_M \ot T^* M)$ and $\Na^{\sE_B} : \Ga^\infty(B,E_B) \to \Ga^\infty(B,E_B \ot T^* B)$.

We let 
\[
\Cl_V(M) := \Ga^\infty\big(M, \Cl(T_V M)\big) , \q  \Cl_H(M) := \Ga^\infty\big(M, \Cl(T_H M)\big)
\]
denote the Clifford algebras of vertical and horizontal vector fields, respectively. Remark that the horizontal vector fields $\sX_H(M)$ are defined as the smooth sections of the smooth vector bundle $T_H M \to M$ with fiber $(T_V M)_x^\perp \subseteq (T M)_x$ at each $x \in M$.

%over $C^\infty(M)$ and $C^\infty(B)$
%
%for which we have
%\end{align*}
%
%the natural $C^\infty(M)$ and $C^\infty(B)$ spinor modules to consider are then $\sE_M$ and $\sE_B$, respectively, without restricting the support of the spinors or with a fall-off condition. The point is that 
% finitely-generated projective spinor modules $\sE_M$ and $\sE_B$ over $C^\infty(M)$ and $C^\infty(B)$, respectively, 

Then, as in \cite{KS16}, we arrive at
\begin{itemize}
\item a $\zz/2\zz$-graded {\em horizontal spinor bundle} $E_H := \pi^* E_B$, together with a hermitian Clifford connection $\nabla^{\sE_H}$ for Clifford multiplication $c_H$ by horizontal vector fields on $M$.
\item a $\zz/2\zz$-graded {\em vertical spinor bundle} $E_V :=  E_H^* \ot_{\Cl( T_H M)} E_M$, together with a hermitian Clifford connection $\nabla^{\sE_V}$ for Clifford multiplication $c_V$ by vertical vector fields on $M$.
\end{itemize}
The explicit formulae for these operations can be found in \cite[Section 3]{KS16}. We let $\sE_H := \Ga^\infty(M,E_H)$ and $\sE_V := \Ga^\infty(M,E_V)$ denote the smooth sections of the horizontal and vertical spinor bundle, respectively.
%For these finitely-generated projective $C^\infty(M)$-modules one can then show that 
%\begin{equation}
%\label{eq:hor-vert-spinors}
%\sE_H  \otimes_{C^\infty(M)} \sE_V \cong \sE_M.
%\end{equation}
%as desired.

The \emph{vertical Dirac operator $\sD_V : \sE_V \to \sE_V$} is defined by the local expression
\[
\sD_V(\xi) = i  \sum_{j=1}^{\dim (F)} c_V(e_j) \nabla^{\sE_V}_{e_j}(\xi) \qquad \xi \in \sE_V \, ,
\]
where $\{ e_j\}$ is a local orthonormal frame of real vertical vector fields. Clearly, $\sD_V$ is an odd first-order differential operator, which only differentiates in the fiber direction.
%
%The following local expression defines a symmetric unbounded operator $S_0 : \cX \to X$:
%\[
%S_0(\xi) = i  \sum_{j=1}^{\dim (F)} c_V(e_j) \nabla^{\sE_V}_{e_j}(\xi) \qquad \xi \in \cX.
%\]

As in Section \ref{ss:symreg} we let $X$ denote the $\zz/2\zz$-graded $C^*$-correspondence from $C_0(M)$ to $C_0(B)$ obtained as the $C^*$-completion of the \emph{compactly supported} sections in $\sE_V$ (the inner product is defined in Equation \eqref{eq:fibinn} and the grading is induced by the grading on $\sE_V$). We promote $\sD_V$ to an odd unbounded operator
\[
(D_V)_0 : \sE_V^c \to X \, .
\]
%The grading operator $\ga_V$ on $\sE_V$ induces a grading operator $\ga_X$ on $X$

As in \cite[Lemma 15]{KS16} we obtain the following symmetry result:

\begin{lma}\label{l:symmetry}
The odd unbounded operator $(D_V)_0 : \sE_V^c \to X$ is symmetric.
\end{lma}
%
%Suppose that $\dim B \leq 2$ or that the curvature $\Omega$ of the bundle ${\overline M} \to B$ vanishes. Then 
%The following local expression defines a symmetric unbounded operator $S_0 : \cX \to X$:
%\[
%S_0(\xi) = i  \sum_{j=1}^{\dim (F)} c_V(e_j) \nabla^{\sE_V}_{e_j}(\xi) \qquad \xi \in \cX.
%\]
%where $\{ e_j\}$ is a local orthonormal frame of real vertical vector fields. 
The closure of $(D_V)_0$ will be denoted by $D_V : \dom(D_V) \to X$. 

\begin{prop}\label{p:unbkas}
The triple $(C^\infty_c(M),X, D_V)$ is an even unbounded Kasparov module from $C_0(M)$ to $C_0(B)$.
\end{prop}
\proof
Let $f \in C^\infty(M)$ and let $x \in M$. By Lemma \ref{l:symmetry} and Theorem \ref{thm:vertical-KK-cycle} it is enough to verify that the principal symbol $\sigma_{\sD_V}(df)(x) : (E_V)_x \to (E_V)_x$ is invertible whenever $(d_V f)(x) : (T_V M)_x \to \C$ is non-trivial. We compute this principal symbol to be given by the local expression
\[
\sigma_{\sD_V}(df) = [\sD_V,f] =  i \sum_{j = 1}^{\dim(F)} c_V(e_j) e_j(f) \, .
\]
Hence $\sigma_{\sD_V}(df)(x) = c_V( (d_V f)^\sharp)(x)$, where 
\[
\sharp : \Tex{Hom}_{C^\infty(M)}( \sX_V(M),C^\infty(M)) \to \sX_V(M) \qquad \sharp : \bra{e_j} \mapsto e_j
\]
denotes the musical isomorphism. This proves the proposition. 
\endproof

We apply the notation $L^2(\sE^c_M)$ and $L^2(\sE^c_B)$ for the $\zz/2\zz$-graded Hilbert space completions of the smooth compactly supported sections of the spinor bundles $E_M$ and $E_B$, respectively. The inner products come from the Riemannian metrics and the hermitian forms in the usual way.

For later use, we record the following result. The proof is the same as the proof of \cite[Proposition 14]{KS16}. Notice, when reading the statement, that we are tacitly applying the identifications $\Tex{End}_{C^\infty(M)}(\sE_H) \cong \Cl_H(M) \subseteq \Cl(M) \cong \Tex{End}_{C^\infty(M)}(\sE_M)$.

\begin{prop}\label{p:uniisom}
The even left $C^\infty_c(M)$-module isomorphism 
\begin{align*}
W : \sE_V^c \otimes_{C^\infty_c(B)} \sE^c_{B} &\to \sE_{M}^c \\
W : ( \bra{\xi} \ot s) \ot r &\mapsto \big( \ket{r \ci \pi} \bra{\xi} \big)(s)
\end{align*}
defined for $\xi \in \sE_H^c$, $s \in \sE_{M}^c$ and $r \in \sE_{B}^c$, induces an even unitary isomorphism 
\[
W : X \hot_{C_0(B)} L^2(\sE_{B}^c) \to L^2(\sE_M^c)
\]
of $\zz/2\zz$-graded Hilbert spaces.
\end{prop}

%{\red I changed the notation here: There are too many $V$'s around.}

%As in 
% that is to say, 
%\[
%\sE_V^c= \sE_V \otimes_{C^\infty(M)} C^\infty_c(M) \, . 
%\]
%Hence $\sE_V^c$ is a $C_c^\infty(M)$-submodule of $\sE_V$ and the induced hermitian structure takes values in $C^\infty_c(M)$. In fact, since $\sE_V^c$ is a $C^\infty_c(M)$-module, it becomes a $C^\infty_c(B)$-module (denoted $\cc X$) using the pullback map $\pi^* : C^\infty_c(B) \to C^\infty_c(M)$. Moreover, $\cc X$ can be equipped with a $C^\infty_c(B)$-valued inner product $\langle \cdot,\cdot\rangle$ defined by integration along the fibers:
%\begin{align*}
%\langle \phi_1,\phi_2 \rangle_{\cc X} (b) := \int_{F_b} \langle \phi_1,\phi_2 \rangle_{\sE_V} \, d\mu_{F_b} 
%\end{align*}
%using the measures coming from the Riemannian metrics on the Riemannian submanifolds of $M$: $F_b := \pi^{-1}(\{b\})$, $b \in B$. We define the Hilbert $C^*$-module $X$ to be the completion of $\cc X$ in the norm coming from this inner product and the $C^*$-norm on $C_0(B)$. There is a left action of $C_0(M)$ on $X$ (coming from the left action of $C^\infty_c(M)$ on $\sE_V^c$) and this gives $X$ the structure of a $C^*$-correspondence from $C_0(M)$ to $C_0(B)$. Moreover, $X$ is $\zz/2\zz$-graded with grading operator $\ga_X$ induced by the grading operator $\ga_V$ on $\sE_V$. The following result is an easy consequence of Equation \eqref{eq:hor-vert-spinors}.

%{\blu We should show that this is a regular closed symmetric operator, and that the bounded transform of $S$ gives rise to a class in $KK(C_0(M),C_0(B))$.
%}

\subsection{Lift of the Dirac operator on the base}\label{ss:lift}

Our next ingredient is the Dirac operator on the base defined by the local expression,
\[
\sD_B = i  \sum_{\al = 1}^{\dim (B)} c_B( f_\alpha) \nabla^{\sE_{B}}_{f_\alpha} : \sE_{B} \to \sE_B \, ,
\]
for any local orthonormal frame $\{ f_\al\}$ for $\sX(B)$ consisting of real vector fields. This Dirac operator is an odd, symmetric and elliptic, first-order differential operator. We are interested in the associated symmetric unbounded operator
\[
(D_B)_0 : \sE^c_B \to L^2(\sE^c_B)
\]
and we denote its closure by $D_B: \dom(D_B) \to L^2(\sE_B)$. Since the Riemannian manifold $B$ is \emph{not assumed} to be complete, it can happen that $D_B$ is \emph{not selfadjoint} and the triple $(C_c^\infty(B), L^2(\sE_B^c), D_B)$ is therefore in general \emph{not a spectral triple} over $C_0(B)$. Instead, as in \cite{BDT89} and \cite{Hil10}, the above triple forms an even \emph{half-closed chain}, representing the fundamental class $[B] \in KK_0(C_0(B),\C)$.

%As already said, $(D_B)_0$ is a symmetric operator and we denote its closure by $D_B: \dom(D_B) \to L^2(\sE_B)$. 

In order to form the unbounded Kasparov product of the vertical and the horizontal components we need to lift the Dirac operator $D_B$ to a symmetric unbounded operator on the Hilbert space $X \hot_{C_0(B)} L^2(\sE_B^c)$. It turns out that the hermitian Clifford connection $\nabla^{\sE_V}$ on $\sE_V$ does \emph{not define} a metric connection on $\sE_V^c \subseteq X$, due to correction terms that come from the measure on the fibers $M_b$, $b \in B$. %{\red Thus, we no not have the identity $\Na^{\sE_V}(\xi)$ .}

To obtain a metric connection for the $C^\infty_c(B)$-valued inner product $\langle\cdot,\cdot \rangle_X$ on $\sE^c_V$, we recall that the {\em second fundamental form} $S \in \Ga^\infty(M, T^*_V M \ot T^*_V M \ot T^*_H M)$ can be defined by 
\begin{equation}
\label{eq:2nd-fund-form}
S(X,Y,Z) := \frac{1}{2}\big( Z (\inn{X,Y}_M) - \inn{[Z,X],Y}_M - \inn{[Z,Y],X}_M \big)
%\binn{\nabla_{Z_H}^V (X_V) - [ Z_H, X_V],Y_V}_M
\end{equation}
for real vertical vector fields $X,Y$ and real horizontal vector fields $Z$ on $M$. Moreover, the \emph{mean curvature} $k \in \Tex{Hom}_{C^\infty(M)}\big( \sX_H(M),C^\infty(M)\big)$ is given as the trace
\[
k = (\tr \otimes 1) (S) \, .
\]
%(where $\Om^1_V(M) := \Tex{Hom}_{C^\infty(M)}(\sX_V(M), C^\infty(M))$ and $\Om^1_H(M) := \Tex{Hom}_{C^\infty(M)}(\sX_H(M), C^\infty(M))$).
 
As in \cite[Definition 18]{KS16}, we define a metric connection on $\sE_V^c \subseteq X$ by 
\[
\nabla^X_Z(\xi) = \nabla^{\sE_V}_{Z_H}(\xi) + \frac{1}{2} k(Z_H) \cdot \xi \in \sE_V^c \subseteq X
\]
for any real vector field $Z$ on $B$, with horizontal lift $Z_H \in \sX_H(M) \cong \Ga^\infty(M,\pi^*T B)$, and any $\xi \in \sE_V^c \subseteq X$.

\begin{lma}
The local expression 
\[
\begin{split}
& (1 \otimes_\nabla D_B)_0 (\xi \otimes r) \\
& \q := \xi \otimes D_B( r) + i \sum_\alpha \nabla^X_{f_\alpha} (\xi)  \otimes c_B(f_\alpha)(r) \qquad
\xi \in \sE_V^c \, , \, \, r \in \sE_B^c
\end{split}
\]
defines a symmetric unbounded operator 
\[
(1 \ot_\Na D_B)_0 : \sE_V^c \otimes_{C^\infty_c(B)} \sE_{B}^c \to X \hot_{C_0(B)} L^2(\sE_{B}^c) \, .
\]
\end{lma}
We denote the closure of $(1 \ot_\Na D_B)_0$ by
\[
1 \ot_\Na D_B : \dom(1 \ot_\Na D_B) \to X \hot_{C_0(B)} L^2(\sE^c_B) \, .
\]

The selfadjoint and regular unbounded operator $D_V : \dom(D_V) \to X$ induces a selfadjoint unbounded operator
\[
(D_V \ot 1)_0 : \dom(D_V) \ot_{C_0(B)} L^2( \sE^c_B) \to X \hot_{C_0(B)} L^2(\sE^c_B)
\]
and we let $D_V \ot 1 : \dom( D_V \ot 1) \to  X \hot_{C_0(B)} L^2(\sE^c_B)$ denote its closure.

We now compute the commutator of $D_V \otimes 1$ and $1\otimes_\nabla D_B$ (when restricted to $\sE_V^c \otimes_{C^\infty_c(B)} \sE_{B}^c \subseteq X \hot_{C_0(B)} L^2(\sE^c_B)$). This computation will be crucial in the proof of our main Theorem \ref{thm:fact-KK} below.

\begin{lma}\label{lma:commutator-Kuc}
Suppose that $\xi \in \sE^c_V \subseteq X$ and $r \in \sE_B^c \subseteq L^2(\sE_B^c)$. Then we have the local expression
%For all $\xi \ot r \in \cc X \otimes_{C^\infty_c(B)} \sE_{B}^c$ we have
\begin{align*}
&[D_V \otimes 1, 1 \otimes_\nabla D_B](\xi \ot r)\\  \nonumber 
&\quad = -\sum_{j,k,\al} S (e_k, e_j, (f_\alpha)_H)  \big(c_V(e_j) \nabla_{e_k}^{\sE_V} \big) (\xi) \otimes  c_B(f_\alpha)(r)\\ \nonumber 
& \qquad - \sum_{j,\alpha} c_V(e_j)\Big( \Omega^{\sE_V}( e_j, (f_\alpha)_H)+\frac 12  e_j \big( k((f_\alpha)_H) \big) \Big)(\xi)\otimes  c_B(f_\alpha)( r) \, ,
\end{align*}
where $\Omega^{\sE_V} : \Ga^\infty(M,E_V) \to \Ga^\infty(M,E_V \ot T^* M \wedge T^* M)$ is the curvature form of the hermitian connection $\nabla^{\sE_V}$. %\in \Tex{End}_{C^\infty(M)}(\sE_V) \ot_{C^\infty(M)} \Om^2(M)$
\end{lma}
\proof
We let $\Na^V = (P \ot 1) \Na^M P : \Ga^\infty(M,T_V M) \to \Ga^\infty(M, T_V M \ot T^* M)$ denote the compression of the Levi--Civita connection on $M$ to vertical vector fields (thus, $P : \sX(M) \to \sX(M)$ denotes the orthogonal projection with image $\sX_V(M) \subseteq \sX(M)$).

We insert the definition of $D_V$ and $1 \ot_\nabla D_B$ in the above expression and compute
\begin{align*}
& [D_V \otimes 1, 1 \otimes_\nabla D_B]
= - \sum_{j,\alpha} \left[ c_V(e_j) \nabla^{\sE_V}_{e_j} , \nabla_{f_\alpha}^X \right] \otimes  c_B(f_\alpha)\\
& \q =  \sum_{j,\alpha} \Big( c_V(\nabla^V_{(f_\alpha)_H} (e_j)) \nabla^{\sE_V}_{e_j}
- c_V(e_j) \left[ \nabla^{\sE_V}_{e_j},  \nabla_{(f_\alpha)_H}^{\sE_V} \right] \\
& \qqq -\frac{1}{2} c_V(e_j) e_j \big( k((f_\alpha)_H) \big) \Big) \otimes  c_B(f_\alpha)  \, .
\end{align*}
We consider the second term after the last equality sign, for which 
\[
\left[ \nabla^{\sE_V}_{e_j},  \nabla_{(f_\alpha)_H}^{\sE_V} \right] =   \nabla^{\sE_V}_{[e_j,(f_\alpha)_H]} + \Omega^{\sE_V}( e_j, (f_\alpha)_H)
\]
in terms of the curvature form $\Omega^{\sE_V}$ of the connection $\nabla^{\sE_V}$.

For each $\al \in \{1,\ldots,\dim(B)\}$, we proceed by computing the first-order differential operator
%% , and this curvature is a bounded map as well. In other words, 
%% \begin{align*}
%% &  [S \otimes 1, 1 \otimes_\nabla D_B] \\
%% &\quad =  \sum_{j,\alpha}\left(c_V(\nabla^V_{(f_\alpha)_H} (e_j)) \nabla^{\sE_V}_{e_j} + c_V(e_j) \nabla^{\sE_V}_{[(f_\alpha)_H,e_j]}\right)\otimes   c(f_\alpha) \\
%% & \qquad - \sum_{j,\alpha} c_V(e_j)  \left( \Omega^{\sE_V}( e_j, (f_\alpha)_H)+\frac 12  e_j \left( k((f_\alpha)_H) \right) \right)\otimes   c(f_\alpha),
%% \end{align*}
%% where all terms on the second line are bounded. So let us consider the first-order differential operator that appears on the first line:
\begin{align*}
&\sum_j \big( c_V(\nabla^V_{(f_\alpha)_H} (e_j)) \nabla^{\sE_V}_{e_j} + c_V(e_j) \nabla^{\sE_V}_{[(f_\alpha)_H,e_j]} \big)\\
& \qquad = 
\sum_{j,k} c_V(e_j) \left( \binn{\nabla_{(f_\alpha)_H}^M( e_k), e_j}_M + \binn{[(f_\alpha)_H,e_j], e_k}_M \right) \nabla_{e_k}^{\sE_V}\\
&\qquad = 
\sum_{j,k} c_V(e_j) \left(  \frac{1}{2} \binn{ [(f_\alpha)_H, e_k], e_j }_M +\frac{1}{2} \binn{ [(f_\alpha)_H, e_j], e_k }_M\right)  \nabla_{e_k}^{\sE_V} \, ,
\end{align*}
where we used Koszul's formula for the Levi--Civita connection $\Na^M$ on $M$, together with the fact that the Lie-bracket $[ (f_\alpha)_H, e_j]$ is a \emph{vertical} vector field for all $\al,j$, see \cite[Lemma 1]{KS16}. When expressed in terms of the second fundamental form of Equation \eqref{eq:2nd-fund-form} this leads to the desired formula. 
%% \begin{align*}
%% &[S \otimes 1, 1 \otimes_\nabla D_B] \\
%% &\quad = \sum_{j,k} \sS(e_k, e_j, (f_\alpha)_H)  \left(c_V(e_j) \nabla_{e_k}^{\sE_V} \right) \otimes  c(f_\alpha)\\
%% & \qquad - \sum_{j,\alpha} c_V(e_j)\left( \Omega^{\sE_V}( e_j, (f_\alpha)_H)+\frac 12  e_j \left( k((f_\alpha)_H) \right) \right)\otimes  c(f_\alpha).\qedhere
%% \end{align*}
\endproof

As a consequence of the above lemma we obtain that the commutator $[D_V \ot 1, 1 \ot_\Na D_B]$ is relatively bounded by $D_V \ot 1$ \emph{on compact subsets of $M$}.

\begin{lma}\label{l:commesti}
Suppose that $K \subseteq M$ is a compact subset. Then there exists a constant $C > 0$ such that
\[
\| [ D_V \ot 1, 1 \ot_\Na D_B](\eta) \| \leq C \cd \big( \| \eta \| + \| (D_V \ot 1)(\eta) \| \big) \, ,
\]
for all $\eta \in \sE_V^c \ot_{C_c^\infty(B)} \sE_B^c$ with support contained in $K$.
\end{lma}
\begin{proof}
Throughout this proof, we will suppress the left $C^\infty(M)$-module isomorphism $W : \sE_V^c \ot_{C_c^\infty(B)} \sE_B^c \to \sE_M^c$ from Proposition \ref{p:uniisom}.

Without loss of generality, suppose that $K \subseteq U$, where $U \subseteq M$ is an open subset supporting an orthonormal frame $\{e_j\}$ for $(T_V M)|_U$ and where $\pi(U) \subseteq B$ supports an orthonormal frame $\{f_\al\}$ for $(TB)|_{\pi(U)}$.

Recall from Proposition \ref{p:unbkas} that $\sD_V : \sE_V \to \sE_V$ is a vertically elliptic first-order differential operator. In particular, we have that $\sD_V \ot 1$ corresponds to a vertically elliptic first-order differential operator on $\sE_M$. It therefore follows from G\aa rding's inequality that there exists a constant $C_0 > 0$ such that
\[
\sum_{j = 1}^{\dim(F)} \| (\Na^{\sE_V}_{e_j} \ot 1)(\eta) \|
\leq C_0 \cd \big( \| \eta \| + \| (D_V \ot 1)(\eta) \| \big) \, ,
\]
for all $\eta \in \sE_V^c \ot_{C_c^\infty(B)} \sE_B^c$ with support contained in $K \subseteq U$.

By Lemma \ref{lma:commutator-Kuc} there exist $A_1,\ldots,A_{\dim(F)},B \in \Ga^\infty\big(U, \Tex{End}(E_M) |_U \big)$ such that
\[
[D_V \ot 1, 1 \ot_{\Na} D_B](\eta) = \sum_{j = 1}^{\dim(F)} A_j (\Na^{\sE_V}_{e_j} \ot 1)(\eta) + B(\eta) \, ,
\]
for all $\eta \in \sE_V^c \ot_{C_c^\infty(B)} \sE_B^c$ with support contained in $K \subseteq U$.

This proves the present lemma.
\end{proof}

\subsection{Factorization of the Dirac operator}

The tensor sum we are after is given by the symmetric unbounded operator
\[
\begin{split}
(D_V \ti_{\Na} D_B)_0 & := ( D_V \otimes 1 )_0 + (\ga_X \ot 1)(1 \otimes_\nabla D_B)_0  \\ 
& \qq : \dom(  D_V \ti_\Na D_B)_0 \to X \hot_{C_0(B)} L^2(\sE_{B}^c) \, ,
\end{split}
\]
where the domain is the image of $\sE_V^c \otimes_{C^\infty_c(B)} \sE_{B}^c$ in $X \hot_{C_0(B)} L^2(\sE_{B}^c)$ and where $\ga_X : X \to X$ denotes the $\zz/2\zz$-grading operator on $X$. The closure of the symmetric unbounded operator $(D_V \ti_{\Na} D_B)_0$ will be denoted by $D_V \times_\nabla D_B$.

We are going to compare this tensor sum with the Dirac operator on the spin$^c$ manifold $M$. As mentioned earlier, these two unbounded operators agree up to an explicit error term given by the curvature form of the proper Riemannian submersion $\pi : M \to B$. We recall that this curvature form $\Om \in \Ga^\infty( M, T_H^* M \we T_H^* M \ot T_V^* M)$ is defined by
\[
\Om(X,Y,Z) := \inn{ [X,Y],Z}_M
\]
for any real horizontal vector fields $X,Y$ and any real vertical vector field $Z$. We represent this curvature form as an endomorphism of $\sE_M$ via the Clifford multiplication $c_M : \sX(M) \to \Tex{End}_{C^\infty(M)}(\sE_M)$ as follows:
\[
\begin{split}
& c : \Ga^\infty( M, T_H^* M \we T_H^* M \ot T_V^* M) \to \Tex{End}_{C^\infty(M)}(\sE_M) \\ 
& c( \om_1 \we \om_2 \ot \om_3) :=\big[ c_M(\om_1^\sharp), c_M(\om_2^\sharp)\big] \cd c_M(\om_3^\sharp) \, .
\end{split}
\]
Notice that the sharps refer to the musical isomorphisms $\sharp : \Om^1_H(M) \to \sX_H(M)$ and $\sharp : \Om^1_V(M) \to \sX_V(M)$. We emphasize that the corresponding operator $c(\Om) : \sE_M^c \to L^2(\sE_M^c)$ can be \emph{unbounded}.

We recall that the Dirac operator on $M$ is defined by the local expression
\[
\sD_M = \sum_{k = 1}^{\dim(M)} c_M\big( (dx_k)^\sharp \big) \Na^{\sE_M}_{\pa/\pa x_k} : \sE_M \to \sE_M  \, .
\]
As usual, we let $D_M : \dom(D_M) \to L^2(\sE_M^c)$ denote the closure of the symmetric unbounded operator $(D_M)_0 : \sE_M^c \to L^2(\sE_M^c)$ induced by $\sD_M$. 

Let $\ga_B : L^2(\sE_B^c) \to L^2(\sE_B^c)$ denote the grading operator on $L^2(\sE_B^c)$. The grading operator on $X \hot_{C_0(B)} L^2(\sE_B^c)$ is then given by $\ga := \ga_X \ot \ga_B$. We define the even selfadjoint unitary isomorphism
\[
\Ga :=  (\ga_X \ot 1) \frac{ 1 + \ga }{2} + \frac{1 - \ga}{2} : X \hot_{C_0(B)} L^2(\sE_B^c)
\to X \hot_{C_0(B)} L^2(\sE_B^c) \, .
\]

\begin{prop}
\label{prop:tensor-sum}
Under the even unitary isomorphism given by $ W \Ga: X \hot_{C_0(B)} L^2(\sE_{B}^c) \to L^2(\sE_M^c)$ we have the identity
\begin{equation}\label{eq:tensum}
W \Ga ( D_V \times_\nabla D_B) \Ga W^* = \ov{ (D_{M})_0 - \frac{i}{8} c(\Omega) }  \, .
\end{equation}
%
% and where $c(\Omega)$ is Clifford multiplication by the curvature $\Omega \in \Omega^2(M) \otimes \Omega^1(M)$ of the fiber bundle $\pi:M \to B$. 
\end{prop}
\proof
 We first notice that
\[
\Ga ( D_V \times_\nabla D_B) \Ga (\eta) = (D_V \ot \ga_B)(\eta) + (1 \ot_\Na D_B)(\eta) \, , 
\]
for all $\eta \in \sE_V^c \ot_{C_c^\infty(B)} \sE^c_B$. As in the proof of \cite[Theorem 23]{KS16} we then establish that
\[
W \Ga ( D_V \times_\nabla D_B) \Ga W^*(\xi) = \big(  D_{M} - \frac{i}{8} c(\Omega) \big)(\xi)
\]
for all $\xi \in \sE_M^c$. The result of the proposition now follows since $\sE_M^c$ is a core for both of the unbounded operators appearing in Equation \eqref{eq:tensum}.
\endproof

\begin{lma}\label{l:funcla}
The triple $\big( C_c^\infty(M), X \hot_{C_0(B)} L^2( \sE_B^c), D_V \ti_\Na D_B \big)$ is an even half-closed chain over $C_0(M)$. Moreover, this even half-closed chain represents the fundamental class $[M]$ in $KK_0(C_0(M),\C)$.
\end{lma}
\begin{proof}
By Proposition \ref{prop:tensor-sum}, $(D_V \ti_\Na D_B)_0$ is unitarily equivalent to the odd, symmetric and elliptic, first-order differential operator $(D_M)_0 - \frac{i}{8} c(\Omega) : \sE_M^c \to L^2(\sE_M^c)$. This establishes that the triple 
\[
\big( C_c^\infty(M), X \hot_{C_0(B)} L^2( \sE_B^c), D_V \ti_\Na D_B \big)
\]
is an even half-closed chain, see \cite{BDT89,Hil10}.

To end the proof, we need to show that the even half-closed chains 
\begin{equation}\label{eq:locbou}
\begin{split}
& \big( C_c^\infty(M), L^2(\sE_M^c), D_M \big)  \q \Tex{and} \\
& \big( C_c^\infty(M), L^2(\sE_M^c), \ov{ (D_M)_0 - \frac{i}{8} c(\Omega)} \big)
\end{split}
\end{equation}
represent the same class in $KK_0(C_0(M),\C)$. 

For a function $x \in C_0(M)$, we let $\inn{x, C_c^\infty(M)} \subseteq C_0(M)$ denote the smallest $*$-subalgebra of $C_0(M)$ containing both $C_c^\infty(M)$ and $x$. We then choose a positive function $x \in C_0(M)$ such that $x \cd C_0(M) \subseteq C_0(M)$ is norm-dense and such that
\[
\begin{split}
& \big( \inn{x, C_c^\infty(M)}, L^2(\sE_M^c), D_M \big) \q \Tex{and} \\
& \big( \inn{x, C_c^\infty(M)}, L^2(\sE_M^c), \ov{ (D_M)_0 - \frac{i}{8} c(\Omega)} \big)
\end{split}
\] 
are still even half-closed chains from $C_0(M)$ to $\C$. Moreover, we may arrange that
\[
x (\ov{ (D_M)_0 - \frac{i}{8} c(\Omega)})x (\xi) = x D_M x(\xi) - \frac{i}{8} x \ov{c(\Om)} x(\xi)
\]
for all $\xi \in \sE_M^c$ and that $\ov{c(\Om)} x$ is a bounded operator on $L^2(\sE_M^c)$. Clearly, the passage from $C_c^\infty(M)$ to $\inn{x, C_c^\infty(M)}$ does not change the corresponding classes in KK-theory. 

We now localize our symmetric unbounded operators with respect to the positive function $x$, obtaining the {\em essentially selfadjoint} unbounded operators $x D_M x : \sE_M^c \to L^2(\sE_M^c)$ and $x D_M x - \frac{i}{8} x \ov{c(\Omega)} x : \sE_M^c \to L^2(\sE_M^c)$, see \cite[Proposition 11]{KS17a}. Moreover, we have that
\begin{equation}\label{eq:bounded}
\begin{split}
& \big( \inn{x, C_c^\infty(M)}, L^2(\sE_M^c), \ov{x D_M x}\big) \q \Tex{and} \\\
& \big( \inn{x, C_c^\infty(M)}, L^2(\sE_M^c), \ov{ x D_M x} - \frac{i}{8} x \ov{ c(\Omega)} x  \big)
\end{split}
\end{equation}
are even spectral triples over $C_0(M)$ and that these spectral triples represent the same classes in KK-theory as our original even half-closed chains in Equation \eqref{eq:locbou}, see \cite[Theorem 13 and Theorem 19]{KS17a}. 

But the two spectral triples in Equation \eqref{eq:bounded} clearly represent the same class in KK-theory since the unbounded selfadjoint operators $\ov{x D_M x}$ and $\ov{x D_M x} - \frac{i}{8} x \ov{c(\Om)} x$ are bounded perturbations of each other. This ends the proof of the lemma. \qedhere

%: \Tex{Dom}( (D_M)_x) \to L^2(\sE_M^c)$ and $\ov{ (D_M)_0 - \frac{i}{8} c(\Omega)}_x : \Tex{Dom}( ) \to L^2(\sE_M^c)$ 
%
%By \cite[Theorem 13]{} we then have that $$Using the results of \cite[]
%
%adding the element $x$ to the $*$ two new half-closed chains still represent the same class 
%
%But this follows since the term $-\frac{i}{8} c(\Omega) : \sE_M^c \to L^2(\sE_M^c)$ that appears there is a locally bounded perturbation of $(D_M)_0 : \sE_M^c \to L^2(\sE_M^c)$ in the sense of \cite{Dun16} (see also \cite[Corollary 22]{KS17a}) and hence does not affect the KK-class.
\end{proof}

\begin{rem}
The result of the above lemma can not be proved directly using the recent work of van den Dungen, since we are lacking an ``adequate approximate identity'' for the unbounded symmetric operator $D_M : \Tex{Dom}(D_M) \to L^2(\sE_M^c)$, see \cite{Dun16}. Instead we rely on the general localization techniques initiated in \cite{Kaa14} and developped further in \cite{Kaa15,KS17a}.
\end{rem}

For any compact subset $K \subseteq M$, we introduce the subspace
\[
L^2( \sE_M^K) := \big\{ \xi \in L^2(\sE_M^c) \mid \psi \cd \xi = \xi \, , \, \, 
\forall \psi \in C_c^\infty(M) \Tex{ with } \psi|_K = 1 \big\} \, .
\]

\begin{lma}\label{l:domten}
Suppose that $K \subseteq M$ is a compact subset. Then
\[
\dom( D_V \ti_\Na D_B) \cap L^2(\sE_M^K) \subseteq \dom( D_V \ot 1) \, .
\]
Moreover, there exists a constant $C_K > 0$ such that
\[
\| (D_V \ot 1)(\eta) \| \leq C_K \cd \big( \| (D_V \ti_\Na D_B)(\eta) \| + \| \eta \| \big)
\]
for all $\eta \in \dom( D_V \ti_\Na D_B) \cap L^2(\sE_M^K)$.
\end{lma}
\begin{proof}
Throughout this proof we will suppress the left $C^\infty_c(M)$-module isomorphism $W : \sE_V^c \ot_{C_c^\infty(B)} \sE_B^c \to \sE_M^c$ from Proposition \ref{p:uniisom}.

Let $L \subseteq M$ be a compact subset. Since $(D_V \ti_\Na D_B)_0 : \sE_M^c \to L^2(\sE_M^c)$ is induced by an elliptic first-order differential operator and since $(D_V \ot 1)_0 : \sE_M^c \to L^2(\sE_M^c)$ is induced by a first-order differential operator we may apply G\aa rding's inequality to find a constant $C_L > 0$ such that
\begin{equation}\label{eq:gaarding}
\| (D_V \ot 1)(\eta) \| \leq C_L \cd \big( \| (D_V \ti_\Na D_B)(\eta) \| + \| \eta \| \big)
\end{equation}
for all $\eta \in \sE_M^c$ with support contained in $L$.

Let now $\eta \in \dom( D_V \ti_\Na D_B) \cap L^2(\sE_M^K)$. Since $\sE_M^c$ is a core for $D_V \ti_\Na D_B$ we may find a sequence $\{\eta_n\}$ in $\sE_M^c$ such that $\eta_n \to \eta$ and $(D_V \ti_\Na D_B)(\eta_n) \to (D_V \ti_\Na D_B)(\eta)$ in the norm on $L^2(\sE_M^c)$. Moreover, since $\eta \in L^2(\sE_M^K)$ we may suppose, without loss of generality, that there exists a compact subset $L \subseteq M$ such that the support of $\eta_n$ is contained in $L$ for all $n \in \nn$.

The result of the lemma now follows from Equation \eqref{eq:gaarding}.
\end{proof} 

Finally, we establish that the tensor sum half-closed chain 
\[
( C_c^\infty(M), X \hot_{C_0(B)} L^2(\sE_B^c), D_V \times_\nabla D_B )
\]
is indeed an unbounded representative of the Kasparov product of the corresponding classes in bounded KK-theory. Note that the relevant Kasparov product and KK-groups are the following:
\[
\hot_{C_0(B)} : KK_0(C_0(M), C_0(B)) \ti KK_0(C_0(B),\C) \to KK_0(C_0(M),\C) \, . 
\]
We are thus going to prove the identity
\[
%\begin{split}
[ X, F_{D_V}] \hot_{C_0(B)} [ L^2(\sE_B^c), F_{D_B}]
= [X \hot_{C_0(B)} L^2(\sE_B^c), F_{D_V \times_\nabla D_B}]
%\end{split}
\]
in $KK_0(C_0(M),\C)$, where $F_D := D(1 + D^* D)^{-1/2} : E \to E$ denotes the bounded transform of a symmetric and regular unbounded operator $D : \Tex{Dom}(D) \to E$.

The proof will be based on a generalization to half-closed chains of a theorem by Kucerovsky \cite[Theorem 13]{Kuc97}, which we proved recently in \cite{KS17a} (see Appendix \ref{app:kucerovsky} Theorem \ref{t:kuce} for the main result). %In the present case, it states that $( C_c^\infty(M), X \hot_{C_0(B)} L^2(\sE_B^c), D_V \times_\nabla D_B )$ is the unbounded Kasparov product of the unbounded Kasparov module $(C^\infty_c(M), X, D_V)$ with the half-closed chain $(C^\infty_c(B), L^2(\sE_B^c), D_B)$, provided the connection and local positivity condition hold (see Appendix \ref{app:kucerovsky}).

\begin{thm}
\label{thm:fact-KK}
Suppose that $\pi : M \to B$ is a proper Riemannian submersion of even dimensional spin$^c$ manifolds. Then the even half-closed chain $(C^\infty_c(M), L^2(\sE_M^c), D_M)$ is the unbounded Kasparov product of the even unbounded Kasparov module $(C^\infty_c(M), X, D_V)$ with the even half-closed chain $(C^\infty_c(B), L^2(\sE_B^c), D_B)$ up to the curvature term $- \frac{i}{8} c(\Omega) : \sE_M^c \to L^2(\sE_M^c)$. 
\end{thm}
\proof
From Proposition \ref{prop:tensor-sum} we know that the tensor sum half-closed chain is unitarily equivalent to the half-closed chain
\[
\big(C_c^\infty(M), L^2(\sE_M^c), \ov{ (D_M)_0 - \frac{i}{8} c(\Om) } \big) \, .
\]
Moreover, Lemma \ref{l:funcla} says that the tensor sum half-closed chain represents the fundamental class $[M]$ in $KK_0( C_0(M), \C)$. We therefore only need to verify the connection condition (Definition \ref{defn:conn-cond}) and the local positivity condition (Definition \ref{defn:loca-subset}) for the tensor sum half-closed chain, the vertical unbounded Kasparov module and the horizontal half-closed chain.

For the connection condition we work with the core $\sE_V^c \ot_{C_c^\infty(B)} \sE_B^c$ for $D_V \ti_\Na D_B$ and the core $\sE_B^c$ for $D_B$. We then compute locally for homogeneous $\xi \in \sE_V^c \subseteq X$ and $r \in \sE_B^c \subseteq L^2(\sE_B^c)$ that 
\[
\begin{split}
& (D_V \ti_\Na D_B) (\xi \ot r) -  (-1)^{\pa \xi} \cd ( \xi \ot D_B(r) ) \\
& \q = D_V \xi \otimes r +  (-1)^{\pa \xi} \cd i \sum_{\alpha} \nabla_{f_\alpha}^X(\xi) \otimes c_B(f_\alpha) r
\end{split}
\]
which clearly extends to a bounded operator from $L^2(\sE_B^c)$ to the interior tensor product $X \hot_{C_0(B)} L^2(\sE_B^c)$. 

For the localizing subset $\Lambda \subseteq C_c^\infty(M)$ we start by choosing a countable open cover $\{ U_m\}$ of the base manifold $B$ and a smooth partition of unity $\{ \chi_m\}$ subordinate to that cover and with $\supp (\chi_m)$ compact for each $m \in \nn$. Clearly, $\{ \pi^{-1}(U_m) \}$ is a countable open cover of $M$, and $\{ \chi_m \ci \pi\}$ is a partition of unity. We will then take as a localizing subset $\Lambda= \{ \chi_m \ci \pi\}$ for which one readily checks the two first conditions of Definition \ref{defn:loca-subset}. Note that $K_m := \supp (\chi_m \ci \pi) = \pi^{-1}( \supp(\chi_m))$ is compact because $\pi : M \to B$ is assumed to be proper. The final condition of Definition \ref{defn:loca-subset} follows from Lemma \ref{l:domten}.

Let $m \in \nn$ and choose a compact subset $L_m \subseteq M$ such that $K_m$ is contained in the interior of $L_m$. To verify the local positivity condition it suffices to show that there exists a $\kappa_m >0$ such that 
\begin{equation}\label{eq:locpos}
\inn{ (D_V \hot 1) \eta, (D_V\times_\nabla D_B)  \eta} + \inn{ (D_V\times_\nabla D_B) \eta, (D_V \hot 1) \eta} \geq - \ka_m \inn{\eta,\eta}
\end{equation}
for all $\eta \in \Tex{Im}\left((\chi_m \ci \pi)^2 \right) \cap \Tex{Dom}(D_V \times_\nabla D_B)$. 

By Lemma \ref{l:commesti} we may find a constant $C_m > 0$ such that
\begin{equation}\label{eq:commineq}
\| [ D_V \ot 1, 1 \ot_\Na D_B](\eta) \|^2 \leq C_m \cd \big( \| \eta \|^2 + \| (D_V \ot 1) (\eta) \|^2 \big)
\end{equation}
for all $\eta \in \sE_V^c \ot_{C_c^\infty} \sE_B^c$ with support contained in $L_m$.

We claim that the inequality in Equation \eqref{eq:locpos} is satisfied for $\ka_m := \frac{1}{2}(1 + C_m)$.

Suppose first that $\eta \in \sE_V^c \otimes_{C_c^\infty(B)} \sE_B^c$ with $\supp(\eta) \subseteq K_m$. Arguing just as in the proof of \cite[Lemma 7.5]{KL12} and using Equation \eqref{eq:commineq} we obtain that
\begin{align*}
& \pm \inn{  \eta,  (\ga_X \ot 1 )[D_V \ot 1,1\otimes_\nabla D_B]  \eta} \\
& \qquad = \pm \frac{1}{2}\big(  \binn{ C_m^{1/2} \eta,  (\gamma_X \ot 1) [D_V \ot 1,1\otimes_\nabla D_B]   C_m^{-1/2} \eta }
\\
&\qquad \qquad  \quad +  \binn{   (\ga_X \ot 1) [D_V \ot 1,1\otimes_\nabla D_B] C_m^{-1/2}  \eta , C_m^{1/2} \eta }\big)\\
&\qquad \leq \frac{1}{2 C_m} \cd \| [D_V \ot 1,1\otimes_\nabla D_B] \eta\|^2 + \frac{C_m}{2} \cd \|\eta\|^2 \\
& \qquad \leq \frac{1}{2} \cd \| (D_V \ot 1) \eta \|^2 + \ka_m \cd \|\eta\|^2 \, .
%&\qquad \leq \mu^2 \left\| [S \ot 1,1\otimes_\nabla D_B](1+S \ot 1)^{-1}  \eta \right \|^2 \|( 1+S\ot 1 ) \eta\|^2 + \mu^{-2} \inn{\eta,\eta}
\end{align*}
%From Lemma \ref{lma:commutator-Kuc} we may find constants $a_\alpha, b_\alpha$ such that 
%\[
%\left\| [S \ot 1,1\otimes_\nabla D_B] \eta \right \|^2 \leq a_\alpha \| \eta \|^2 + b_\alpha\| (S \ot 1)\eta\|^2
%\]
%for all $\eta \in \cX \otimes_{C_c^\infty(B)} \sE_B^c$ with $\supp(\eta) \subset K_\alpha$. If we then let $\mu^2 = 1/b_\alpha$ we arrive at
This implies that
\begin{equation}\label{eq:locpossmo}
\begin{split}
& \inn{ (D_V \ot 1) \eta, (D_V\times_\nabla D_B)  \eta} + \inn{ (D_V\times_\nabla D_B) \eta, (D_V \ot 1) \eta}  \\
&  \q = \inn{  \eta,  (\ga_X \ot 1 )[1\otimes_\nabla D_B,D_V \ot 1]  \eta}
+  2 \cd \| (D_V \ot 1) \eta \|^2  \\
& \q \geq - \kappa_m \| \eta\|^2  
\end{split}
\end{equation}
and hence that Equation \eqref{eq:locpos} holds for all $\eta \in \sE_V^c \otimes_{C_c^\infty(B)} \sE_B^c$ with $\supp(\eta) \subseteq L_m$.

Suppose now that $\eta \in \Tex{Im}\left((\chi_m \ci \pi)^2 \right) \cap \Tex{Dom}(D_V \times_\nabla D_B)$. We then choose a sequence $\{ \eta_n\}$ in $\sE_V^c \otimes_{C_c^\infty(B)} \sE_B^c$ which converges to $\eta$ in the graph norm of $D_V \ti_\Na D_B$. We may assume, without loss of generality, that $\eta_n$ has support in $L_m \subseteq M$ for all $n \in \nn$. By Lemma \ref{l:domten} this implies that $\{ \eta_n\}$ also converges to $\eta$ in the graph norm of $D_V \ot 1$. The inequality in Equation \eqref{eq:locpos} therefore follows from Equation \eqref{eq:locpossmo}.

%We will first show that it is sufficient to check this for $\eta$ in a core for $S \hot 1$ (and in the image of $(\pi^* \chi_\alpha)^2$). Indeed, let $\eta_n \to \eta$ in the graph norm for $S \hot 1$ with $\eta_n \in \cX \otimes_{C_c^\infty(B)} \sE_B^c$. We may assume (possibly after multiplying by $\pi^*\chi_\alpha$) that $\supp (\eta_n) \subset K_\alpha$. Then
%\begin{align*}
%&\inn{ (S \ot \gamma_B) \eta_n, (S\times_\nabla D_B)  \eta_n} + \inn{ (S\times_\nabla D_B) \eta_n, (S \ot \gamma_B) \eta_n} \\
%&\quad = \inn{ (S \ot 1) \eta_n, (S \ot 1) \eta_n}  +\inn{  \eta_n,\gamma_B [S \ot 1, 1\otimes_\nabla D_B]  \eta_n} \\
%&\quad = \inn{ (S \ot 1) \eta_n, (S \ot 1) \eta_n} \\ &\qquad\qquad  
%+ \inn{  \eta_n,  \gamma_B[S \hot 1,1 \otimes_\nabla D_B] (1+S\ot 1)^{-1} (1+S \ot 1) \eta_n} \\
%& \quad \to \inn{ (S \ot 1) \eta, (S \ot 1) \eta} +\inn{  \eta, \gamma_B [S \hot 1,S\times_\nabla D_B] \eta} 
%\end{align*}
%as $n \to \infty$. This follows because, when restricted to sections with support in a compact subset $K_\alpha$, we may apply Lemma \ref{lma:commutator-Kuc} to derive that the commutator $[S \otimes 1, 1 \otimes_\nabla D_B]$ is relatively bounded with respect to $S \ot 1$. 

We have thus established the local positivity condition and this completes the proof of the theorem. 
\endproof

\section{Almost regular fibrations}

%{\red I removed the compactness assumption on $\ov M$ in the following.}

In this section we come to the third main result of this paper, which is the factorization in unbounded KK-theory of the Dirac operator $\sD_{\ov M}$ on the total manifold $\ov M$ of a so-called almost-regular fibration of spin$^c$ manifolds. This factorization takes place on a dense open submanifold $M$ of $\ov M$ and is given in terms of a vertical Dirac operator $\sD_V$ and the Dirac operator $\sD_B$ on a base manifold $B$. The point is here that all information about the total Dirac operator $\sD_{\ov M}$ can be deduced from the behaviour of its restriction $\sD_M$ to $M$. But let us begin with the precise definitions.

\begin{defn}
Let $\overline M$ be a Riemannian manifold (not necessarily compact but without boundary), together with a finite union $P = \cup_{j = 1}^m P_j$ of compact embedded submanifolds $P_j \subseteq \overline M$ each without boundary and of codimension strictly greater than $1$. If there exist a Riemannian manifold without boundary $B$ and a proper Riemannian submersion $\pi : M \to B$ with total space $M= \overline M \sem P$, we call the data $(\overline M,P,B,\pi)$ an {\em almost-regular fibration}.
\end{defn}

Note that $M = \overline M \sem P$ is a dense open subset of $\overline M$ (and $P$ has Riemannian measure zero).

\begin{ex}
A proper Riemannian submersion $\pi: M \to B$ is an example of an almost-regular fibration when we take $P =\emptyset$. 
\end{ex}

\begin{ex}
\label{ex:action} %{\red THIS IS STILL NOT CORRECT!!! THERE COULD BE EXCEPTIONAL ORBITS!!!} 
Let $G$ be a torus acting isometrically, but not necessarily freely, on a compact Riemannian manifold $N$ such that the orbit space $N/G$ is connected. Let $H_{\Tex{prin}}, H_1,\ldots,H_m \subseteq G$ denote the finitely many isotropy groups, where $H_{\Tex{prin}}$ is the principal stabilizer. We suppose that all the orbits $G/H_j$, $j = 1,\ldots,m$, are singular and moreover, that each subspace of $H_j$-fixed points $N^{H_j}$ is connected. Then, letting $N_0 \subseteq N$ denote the principal stratum, we have that $N\sem N_0 = \cup_{j = 1}^m N^{H_j}$ and that each $N^{H_j} \subseteq N$ is a compact embedded submanifolds of codimension strictly greater than $1$, see \cite[Theorem 5.11, Theorem 5.14 and Proposition 5.15]{Die87} and \cite[Proposition 1.24]{Mei03}. The projection map $\pi: N_0 \to N_0/G$ is a proper Riemannian submersion hence the data $(N, \cup_{j=1}^m N^{H_j}, N_0/G, \pi)$ is an almost-regular fibration.

In fact, in the above we may restrict our attention to a subset $\{ H_{j_i} \}$ of the isotropy groups such that $N \sem N_0 = \cup_i N^{H_{j_i}}$. It does moreover suffice to assume that each quotient space $N^{H_{j_i}}/G$ is connected instead of assuming that each $N^{H_{j_i}}$ is connected. In this case we use the connected components of the $H_{j_i}$-fixed points as our compact embedded submanifolds instead.
\end{ex}
%
%there is a partition $N = N_0 \cup P_1 \cup \cdots \cup P_N$ into $G$-orbit types where the principal orbit type $N_0$ is a dense open subset  of $N$ while the next orbit types $P_j$ are compact embedded submanifolds of codimension strictly greater than $1$ 

\begin{defn}
An {\em almost-regular fibration of spin$^c$ manifolds} is an almost-regular fibration $(\overline M,P,B,\pi)$ such that $\overline M$ and $B$ are equipped with spin$^c$ structures.
\end{defn}

\begin{ex}
\label{ex:action-spinc}
Continuing with Example \ref{ex:action}, if we assume that $N$ carries a $G$-equivariant spin$^c$ structure, it follows that the open submanifold $N_0$ is a $G$-equivariant spin$^c$ manifold. If in addition the action of $G$ on $N$ is effective, the orbit space $N_0/G$ is a spin$^c$ manifold. This forms a key class of examples of almost-regular fibrations of spin$^c$ manifolds. 
\end{ex}

Let us consider an almost-regular fibration of even dimensional spin$^c$ manifolds $(\ov M, P, B,\pi)$. We let $E_{\ov M} \to \ov M$ and $E_B \to B$ denote the $\zz/2\zz$-graded spinor bundles. Thus there exist Dirac operators on the spin$^c$ manifolds $\overline M$ and $B$, given as odd unbounded operators
\[
\begin{split}
& ( D_{\overline M})_0 : \Ga^\infty_c( \ov M, E_{\overline M}) \to L^2( \ov M, E_{\overline M}) \q \Tex{and} \\
& ( D_B)_0 : \Ga^\infty_c(B, E_B) \to L^2(B, E_B) \, .
\end{split}
\]
Moreover, since the spin$^c$ structure on $\ov{M}$ restricts to a spin$^{c}$ structure on the open submanifold $M \subseteq \ov{M}$ we also obtain the Dirac operator
\[
(D_M)_0 : \Ga^\infty_c(M, E_M) \to L^2(M,E_M) \, ,
\]
where the spinor bundle $E_M \to M$ agrees with the restriction of $E_{\ov M}$ to $M \subseteq \ov M$. Since $P \subseteq \ov M$ has Riemannian measure $0$ we may identify the $\zz/2\zz$-graded Hilbert spaces $L^2(M,E_M)$ and $L^2(\ov M, E_{\ov M})$ using the inclusion $\io : \Ga^\infty_c(M, E_M) \to \Ga^\infty_c( \ov M, E_{\overline M})$ given by extension by zero. We may then arrange that
\[
(D_{\ov M})_0( \io (\xi) ) = (D_M)_0(\xi) \q \Tex{for all } \xi \in \Ga^\infty_c(M, E_M) \, .
\]
Each of these unbounded operators determine even half-closed chains 
\[
\begin{split}
& (C^\infty_c(\ov M), L^2(\ov M, E_{\ov M}), D_{\ov M}) \\
& (C^\infty_c(B), L^2(B,E_B), D_B) \q \Tex{and} \\
& (C^\infty_c(M), L^2(M,E_M), D_M) \, ,
\end{split}
\]
representing the fundamental classes $[\ov M] \in KK_0(C_0(\ov M), \C)$, $[B] \in KK_0(C_0(B),\C)$ and $[M] \in KK_0(C_0(M),\C)$, respectively.
%
%Note that, due to the compactness conditions on ${\overline M}$, the unbounded operator $( D_{\overline M} )_0$ is an essentially self-adjoint operator and defines an even spectral triple $(C^\infty(\ov M), L^2(\sE_{\ov M}), D_{\ov M})$ for $C({\overline M})$, representing the fundamental class $[{\overline M}]$ in $KK(C({\overline M}),\C)$. Instead, the unbounded operator $(D_B)_0$ is only symmetric in general and does not define a spectral triple in the usual sense. However, as mentioned in Subsection \ref{ss:lift}, we do obtain an even half-closed chain $(C^\infty_c(B), L^2(\sE^c_B), D_B)$ over $C_0(B)$, representing the fundamental class $[B]$ in $KK(C_0(B),\C)$. {\blu Finally we have the triple $(C^\infty_c(M), L^2(\sE_M^c), D_M)$ which, at least a priori, is again only an even half-closed chain over $C_0(M)$ representing the fundamental class $[M]$ in $$.}

Using the $*$-homomorphism $\io : C_c^\infty(M) \to C^\infty_c(\ov M)$ given by extension by zero, we may pullback the even half-closed chain 
\[
(C^\infty_c(\ov M), L^2(\ov M,E_{\ov M}), D_{\ov M})
\]
to an even half-closed chain
\[
\io^*( C^\infty_c(\ov M), L^2(\ov M, E_{\ov M}), D_{\ov M} ) = ( C_c^\infty(M), L^2(\ov M, E_{\ov M}), D_{\ov M}) \, .
\]
At the level of bounded KK-theory this pullback operation corresponds to the usual pullback homomorphism
\[
\io^* : KK_0( C_0(\ov M), \C) \to KK_0( C_0(M), \C)
\]
coming from the $*$-homomorphism $\io : C_0(M) \to C_0( \ov M)$ and the contravariant functoriality in the first variable.

The following theorem is the third main result of this paper. The proof relies for the main part on Theorem \ref{thm:fact-KK} and we apply the notation of that theorem here as well.

\begin{thm}
\label{thm:almost-regular-fact}
Suppose that $(\ov M, P, B,\pi)$ is an almost regular-fibration of even dimensional spin$^c$ manifolds. Then it holds (up to unitary equivalence) that
\begin{equation}\label{eq:restric}
\io^*( C^\infty_c(\ov M), L^2(\ov M,E_{\ov M}), D_{\ov M} ) = ( C_c^\infty(M), L^2(M,E_M), D_M) \, .
\end{equation}
Moreover, the even half-closed chain $( C_c^\infty(M), L^2(M, E_M), D_M)$ is the unbounded Kasparov product of the even unbounded Kasparov module $(C_c^\infty(M),X,D_V)$ and the even half-closed chain $(C_c^\infty(B), L^2(B,E_B), D_B)$ up to the curvature term $-\frac{i}{8} c(\Om) : \Ga_c^\infty(M,E_M) \to L^2(M,E_M)$. In particular it holds that
\[
\io^*[ \ov M] = [M] = [ X, D_V(1 + D_V^2)^{-1/2}] \hot_{C_0(B)} [B]
\]
at the level of bounded KK-theory. 
\end{thm}
\begin{proof}
Since $\pi : M \to B$ is a proper Riemannian submersion by assumption, we may apply Theorem \ref{thm:fact-KK} to establish the second part of the present theorem. It therefore suffices to prove the identity in Equation \eqref{eq:restric}. But this identity follows immediately from Proposition \ref{prop:diff-op-dense-subset} here below.
\end{proof}

The restriction we impose on the codimension of the compact embedded submanifolds $P_j \subseteq \ov{M}$ in the definition of an almost-regular fibration $({\overline M},P,B,\pi)$ guarantees that given a closed extension of a first-order differential operator ---such as the Dirac operator--- on ${\overline M}$, its restriction to sections that have compact support contained in $M$ gives the same closure. The following lemma will be useful in this context.

\begin{lma}
\label{lma:seq-functions}
Let $P \subseteq {\overline M}$ be a subset of a Riemannian manifold $\ov M$ such that $P = \bigcup_{j =1}^m P_j$ is a finite union of compact embedded submanifolds $P_j \subseteq \ov M$, each of codimension strictly greater than $1$. Then there exists an increasing sequence $\{\psi_n\}$ of positive smooth functions on $\ov M$ such that 
\begin{enumerate}
\item $\supp (\psi_n) \subseteq \ov{M}\setminus P$ for all $n \in \nn$;
%\item $\psi_n (x) \in [0,1]$ for all $x \in {\overline M}$
\item $\sup_n ( \psi_n|_{\overline M \sem P}) = 1_{\overline M \sem P}$;
\item The exterior derivative $d \psi_n$ has compact support for all $n \in \nn$ and the sequence $\{  d \psi_n \}$ is bounded in $L^2(\ov M, T^* \ov M)$.
\end{enumerate}
\end{lma} 
\proof
We may assume, without loss of generality, that $P$ consists of a single closed embedded submanifold $P \subseteq \ov{M}$. Indeed, suppose that an increasing sequence $\{\psi_n^j\}$ of positive smooth functions satisfying $(1)$, $(2)$ and $(3)$ has been constructed for each compact embedded submanifold $P_j \subseteq \ov{M}$. Then the sequence $\{ \psi_n\} := \{ \psi_n^1 \clc \psi_n^m\}$ satisfies $(1)$, $(2)$ and $(3)$ for $P = \cup_{j = 1}^m P_j$.
%Choose an increasing sequence of $\{ g_n \}$ of positive smooth functions on $[0,\infty)$ such that
%\begin{itemize}
%\item $\supp (g_n) \subseteq (0,\infty)$;
%\item $\supp ( d g_n) \subseteq [0,1/n)$;
%\item $\tilde\psi_n (t) \in [0,1]$ for all $t \in [0,\infty)$
%\item $\tilde \psi_n(0) = 0$ and $\tilde \psi_n(t) = 1$ for all $t \in [1/n, \infty)$.
%\item $\sup_n( g_n|_{(0,\infty)} ) = 1|_{(0,\infty)}$; %as $n \to \infty$ pointwise for all $t \in [0,\infty)$.
%\item There exists a constant $C > 0$ such that $\|  \frac{dg_n}{dr} \|_\infty \leq C \cd n$ for all $n \in \nn$.
%\end{itemize}
%
%For each $k \in \nn$ we then define the increasing sequence $\{ \wit{\psi_n}\}$ of positive smooth functions on $\rr^k$ by $\wit{\psi_n}(t) := g_{n^2}( \|t\|^2)$, where $\| \cd \| : \rr^k \to [0,\infty)$ denotes the Euclidean norm. This sequence then has the properties
%\begin{itemize}
%\item $\supp (\wit{\psi_n}) \subseteq \rr^k \sem \{0\}$;
%\item $\supp ( d \wit{\psi_n}) \subseteq \mathbb{B}_{1/n}(0)$, where $\mathbb{B}_{1/n}(0) \subseteq \rr^k$ denotes the ball of radius $1/n$ and center %$0$;
%\item $\sup_n( \wit{\psi_n}|_{\rr^k \sem \{0\}} ) = 1|_{ \rr^k \sem \{0\}}$;
%\item There exists a constant $C > 0$ such that $\|  d \wit{\psi_n}\|_\infty \leq C \cd n$ for all $n \in \nn$, where $\| \cd \|_\infty$ denotes the supremum norm on $\Ga^\infty_c( \rr^k, T^* \rr^k) \cong C^\infty_c( \rr^k)^{\op k}$.
%\end{itemize}

Let $k > 1$ denote the codimension of $P \subseteq \ov{M}$. Choose an increasing sequence $\{ \tilde{\psi_n}\}$ of positive smooth functions on $\rr^k$ such that
\begin{itemize}
\item $\supp (\tilde{\psi_n}) \subseteq \rr^k \sem \{0\}$;
\item $\supp ( d \tilde{\psi_n} ) \subseteq \mathbb{B}_{1/n}(0)$, where $\mathbb{B}_{1/n}(0) \subseteq \rr^k$ denotes the ball of radius $1/n$ and center $0$;
\item $\sup_n( \tilde{\psi_n}|_{\rr^k \sem \{0\}} ) = 1|_{\rr^k \sem \{0\}}$; %as $n \to \infty$ pointwise for all $t \in [0,\infty)$.
\item There exists a constant $C > 0$ such that $\|  d \tilde{ \psi_n} \|_\infty \leq C \cd n$ for all $n \in \nn$, where $\| \cd \|_\infty$ denotes the supremum norm on $\Ga^\infty_c( \rr^k, T^* \rr^k) \cong C^\infty_c( \rr^k)^{\op k}$.
\end{itemize}

Choose a finite open cover of $P \subseteq \ov M$ by submanifold charts
\[
(V^1, \varphi^1), \ldots, (V^N, \varphi^N) \, ,
\]
such that $\ov{V^i} \subseteq \ov M$ is compact for all $i \in \{1,\ldots,N\}$. Thus, for each $i \in \{1,\ldots,N\}$, we have that
\[
V^i \cap P = \big\{ x \in V^i \mid (\pi_1 \ci \varphi^i)(x) = 0 \big\} \, ,
\]
where $\pi_1 : \rr^k \ti \rr^{\dim(\ov M) - k} \to \rr^k$ denotes the projection onto the first $k$ coordinates.
%Notice here that $\de > 0$ is assumed to be a fixed finite radius.
 
%Using the compactness of $P \subseteq \ov M$ (again) we may moreover assume that the closure $\ov{V^j} \subseteq \ov M$ is compact for all $j = 1,\ldots,N$.
%\[
%g^j : V^j \to M_{\dim(\ov M)}(\rr) \q (g^j)_{sr} = \binn{ \pa/ \pa \varphi^j_s , \pa/\pa \varphi^j_r}_{\ov{M}}
%\]
%and its inverse are bounded for all $j \in \{1,\ldots,m\}$. Put $V^0 := \ov{M} \sem P$. 

Put $V^0 := \ov{M} \sem P$ and choose a smooth partition of unity $\chi_0, \chi_1,\ldots,\chi_N$ for $\ov{M}$ with $\supp(\chi_i) \subseteq V^i$ for all $i \in \{0,1,\ldots,N\}$. Remark that $\supp(\chi_i)$ is compact for all $i \in \{1,\ldots,N\}$ but that $\supp(\chi_0)$ need not be compact. It does however hold that the support of the exterior derivative $d \chi_0$ is compact.

We define
\[
\psi_n := \chi_0 + \sum_{i = 1}^N \chi_i \cd (\tilde{\psi_n} \ci \pi_1 \ci \varphi^i ) \q \Tex{ for all } n \in \nn \, .
\]
We leave it to the reader to verify that the increasing sequence $\{\psi_n\}$ of positive smooth functions satisfies $(1)$, $(2)$ and $(3)$. When verifying $(3)$, notice that 
\[
\big\{ \| d \tilde{\psi_n} \|_\infty^2 \cd \Tex{Vol}(\mathbb{B}_{1/n}(0) \subseteq \mathbb{R}^k) \big\}_{n = 1}^\infty
\]
is a bounded sequence since $k > 1$.
\endproof

%\begin{figure}
%\includegraphics[scale=.7]{./psin.png}
%\caption{An example of a family of functions $\tilde\psi_n$ on $[0,\infty)$ satisfying the conditions used in the proof of Lemma \eqref{lma:seq-functions} is given by $\tilde\psi_n (t) = g\left ( \frac{nt-\epsilon}{1-\epsilon} \right)$ where $\epsilon \in (0,1)$ and $g(t)$ is a smooth function which vanishes for $t\leq 0$ and is equal to $1$ for $t \geq 1$. Indeed, then we have that $\|\tilde \psi_n'\| \leq \frac{n}{1-\epsilon} \| g '\|_\infty$.
%}
%\label{fig:examplePsin}
%\end{figure}

The following proposition generalizes the results in \cite[Prop. 4.12]{BLS15} and \cite[Sect. 2.2.1]{FGMR16}.

\begin{prop}\label{prop:diff-op-dense-subset}
Let ${\overline M}$ be a Riemannian manifold and $P = \cup_{j = 1}^m P_j$ be a finite union of compact embedded submanifolds, each of codimension strictly greater than $1$; write the complement as $M := {\overline M}\setminus P$. Let $E \to {\overline M}$ be a smooth hermitian vector bundle and $(D_{\overline M})_0 : \Gamma^\infty_c({\overline M},E) \to L^2({\overline M},E)$ be a first-order differential operator. Then, if we let $(D_{M})_0 : \Gamma_c^\infty(M,E|_M) \to L^2(M,E|_M)$ denote the restriction of $(D_{\overline M})_0$ to the smooth compactly supported sections $\Gamma_c^\infty(M,E|_M)$, the closure of $(D_{M})_0$ coincides with the closure of $(D_{\overline M})_0$, both as operators in $L^2({\overline M},E)$. 
\end{prop}
\proof
First of all, we use that both ${\overline M}$ and $M$ are manifolds without boundary so that by \cite[Lemma 10.2.1]{HR00} both $(D_{\overline M})_0$ and $(D_{M})_0$ are closable. We denote the closures by $D_{\ov M}$ and $D_M$, respectively. Since $P \subseteq {\overline M}$ is a null-set, we can identify $L^2(M,E|_M)$ with $L^2({\overline M},E)$ and consider the operators $D_{\overline M}$ and $D_{M}$ both as being (densely defined) on $L^2({\overline M},E)$. Indeed, the inclusion $\Ga_c^\infty(M,E|_M) \to \Ga_c^\infty(\ov M, E)$ induces a unitary isomorphism of Hilbert spaces $L^2(M,E|_M) \cong L^2(\ov M,E)$.
 
Clearly, $D_M \subseteq D_{\ov M}$, so it suffices to show that the core $\Ga^\infty_c( \ov M,E)$ for $D_{\ov M}$ is included in the domain of $D_M$.

Let thus $s \in \Gamma^\infty_c({\overline M},E)$ be given. For the increasing sequence of positive smooth functions $\{ \psi_n\}$ constructed in Lemma \ref{lma:seq-functions} we have that $\psi_n s \to s$ in the norm of $L^2({\overline M},E)$ (recall that $P$ has Riemannian measure $0$). By construction $\psi_n  s \in \Ga_c^\infty(M,E|_M)$ for all $n \in \nn$ and, by \cite[Lemma 1.8.1]{HR00}, we may thus conclude that $s \in \dom(D_M)$, if we can establish that $\{ D_{M}(\psi_n s) \}$ is a bounded sequence in $L^2({\overline M},E)$. 

Let $\si_{\ov M} : \Ga^\infty(\ov M, T^* \ov{M} ) \to \Ga^\infty(\ov M, \Tex{End}(E))$ denote the principal symbol of the first-order differential operator $\sD_{\ov M} : \Ga^\infty(\ov M,E) \to \Ga^\infty(\ov M,E)$. For each $x \in \ov M$ we let $\| \si_{\ov M}(x) \|_\infty$ denote the norm of the fiber-wise operator $\si_{\ov M}(x) : T_x^* \ov M \to \Tex{End}(E_x)$ and we let $\| s \|_\infty$ denote the supremum norm of the compactly supported section $s : \ov M \to E$. 

We then have the estimates
\begin{align*}
\| D_{M}(\psi_n s)\|_{L^2({\overline M},E)}
&= \| \si_{\ov M}( d \psi_n) \cdot s + \psi_n \cd D_{\overline M} (s)\|_{L^2({\overline M},E)}\\
&\leq \| \si_{\ov M}( d \psi_n) \cdot s \|_{L^2({\overline M},E)} + \| D_{\overline M} (s)\|_{L^2({\overline M},E)} \\
& \leq  \sup_{x \in \supp(s)} \| \si_{\ov M}(x) \|_\infty  \cd \| s \|_\infty \cd \| d \psi_n \|_{L^2(\ov M, T^* \ov M)} \\
& \qqq + \| D_{\overline M} (s)\|_{L^2({\overline M},E)} \, .
\end{align*}
Using property $(3)$ of the sequence $\{ \psi_n \}$ from Lemma \ref{lma:seq-functions}, these estimates show that $\{ \| D_{M}(\psi_n s)\|_{L^2({\overline M},E)} \}$ is a bounded sequence. This proves the proposition.
\endproof

\appendix

\section{On a theorem of Kucerovsky for half-closed chains}
\label{app:kucerovsky}
We summarize the main result of \cite{KS17a}, which generalizes a theorem by Kucerovsky \cite[Theorem 13]{Kuc97} to half-closed chains. Let us fix three $C^*$-algebras $A,B$ and $C$ with $A$ separable and $B$, $C$ $\si$-unital. Throughout this section we will assume that $(\sA,E_1,D_1)$, $(\sB,E_2,D_2)$ and $(\sA,E,D)$ are even half-closed chains from $A$ to $B$, from $B$ to $C$ and from $A$ to $C$, respectively. We denote the $*$-homomorphisms associated to the $C^*$-correspondences $E_1$, $E_2$ and $E$ by $\phi_1 : A \to \mathbb{L}(E_1)$, $\phi_2 : B \to \mathbb{L}(E_2)$ and $\phi : A \to \mathbb{L}(E)$, respectively. We will moreover assume that $E = E_1 \hot_B E_2$ agrees with the interior tensor product of the $C^*$-correspondences $E_1$ and $E_2$. In particular, we assume that $\phi(a) = \phi_1(a) \ot 1$ for all $a \in A$. We let $\ga_1 : E_1 \to E_1$, $\ga_2 : E_2 \to E_2$ and $\ga := \ga_1 \hot \ga_2$ denote the $\zz/2\zz$-grading operators on $E_1$, $E_2$ and $E$.

We will denote the bounded transforms of our half-closed chains by $(E_1, F_{D_1})$, $(E_2,F_{D_2})$ and $(E,F_D)$ and the corresponding classes in KK-theory by $[E_1, F_{D_1}] \in KK_0(A,B)$, $[E_2,F_{D_2}] \in KK_0(B,C)$ and $[E,F_D] \in KK_0(A,C)$. We may then form the interior Kasparov product
\[
[E_1, F_{D_1}] \hot_B [E_2,F_{D_2}] \in KK_0(A,C)
\]
and it becomes a relevant question to find an explicit formula for this class in $KK_0(A,C)$. 

For each $\xi \in E_1$, we let $T_\xi : E_2 \to E$ denote the bounded adjointable operator given by $T_\xi(\eta) := \xi \ot \eta$ for all $\eta \in E_2$.

\begin{defn}
\label{defn:conn-cond}
The {\em connection condition} demands that there exist a dense $\sB$-submodule $\sE_1 \su E_1$ and cores $\sE_2$ and $\sE$ for $D_2 : \Tex{Dom}(D_2) \to E_2$ and $D : \Tex{Dom}(D) \to E$, respectively, such that
\begin{itemize}
\item[(a)] For each $\xi \in \sE_1$:
\[
T_\xi( \sE_2 ) \su \Tex{Dom}(D) \, \, , \q T_\xi^*(\sE) \su \Tex{Dom}(D_2) \, \, , \q \ga_1(\xi) \in \sE_1 
\]
\item[(b)] For each homogeneous $\xi \in \sE_1$:
\[
D T_\xi - (-1)^{\pa \xi} T_\xi D_2 : \sE_2 \to E
\]
extends to a bounded operator $L_\xi : E_2 \to E$.
\end{itemize}
\end{defn}

\begin{defn}
\label{defn:loca-subset}
A {\em localizing subset} is a \emph{countable} subset $\La \su \sA$ with $\La = \La^*$ such that
\begin{itemize}
\item[(a)] The span
\[
\Tex{span}_{\C} \big\{ x \cd a \mid a \in A \, , \, \, x \in \La \big\} \su A
\]
is norm-dense in $A$.
%\item[(b)] $\phi(x)\big( \Tex{Dom}(D^*) \big) \su \Tex{Dom}(D)$ and the commutator
%\[
%[D, \phi(x)] : \Tex{Dom}(D) \to E
%\]
%extends to a bounded operator $d(x) : E \to E$ for all $x \in \La$.
\item[(b)] %$\phi(x)\big( \Tex{Dom}(D_1^* \hot 1) \big) \su \Tex{Dom}(D_1 \hot 1)$ and 
The commutator
\[
[D_1 \hot 1, \phi(x) ] : \Tex{Dom}(D_1 \hot 1) \to E
\]
is  \emph{trivial} for all $x \in \La$.
\item[(c)] We have the domain inclusion
\[
\Tex{Dom}(D) \cap \Tex{Im}(\phi(x^* x)) \su \Tex{Dom}(D_1 \hot 1)
\]
for all $x \in \La$.
\end{itemize}
\end{defn}

\begin{defn}
\label{defn:loca-pos-cond}
Given a localizing subset $\Lambda \su \sA$, the {\em local positivity condition} requires that for each $x \in \La$, there exists a constant $\ka_x > 0$ such that 
\[
\begin{split}
%& \binn{ (D_1 \hot 1) \eta, D \eta } + \inn{ D \eta, (D_1 \hot 1) \eta}  \geq - \ka_x \cd \inn{\eta, \eta} 
& \binn{ (D_1 \hot 1) \phi(x^*) \xi, D \phi(x^*) \xi } + \inn{ D \phi(x^*) \xi, (D_1 \hot 1) \phi(x^*) \xi} \\
& \q \geq - \ka_x \cd \inn{\xi, \xi} 
\end{split}
\]
for all $\xi \in \Tex{Im}(\phi( x)) \cap \Tex{Dom}( D \phi(x^*) )$.
%$\eta \in \Tex{Im}(\phi(x^* x)) \cap \Tex{Dom}( D )$.
%\item $\phi(A) E \su E$ is norm-dense in $E$.
\end{defn}

In practice, it is useful to record that the local positivity condition would follow if for each $x \in \La$, there exists a constant $\ka_x > 0$ such that
\[
\inn{ (D_1 \hot 1) \eta, D \eta} + \inn{D \eta, (D_1 \hot 1) \eta} \geq - \ka_x \inn{\eta,\eta} \, ,
\]
for all $\eta \in \Tex{Im}(\phi(x^* x)) \cap \Tex{Dom}(D)$.

\begin{thm}\label{t:kuce}
Suppose that the three half-closed chains $(\sA,E_1,D_1)$, $(\sB,E_2,D_2)$ and $(\sA,E,D)$ satisfy the connection condition and the local positivity condition. Then $(E,F_D)$ is the Kasparov product of $(E_1,F_{D_1})$ and $(E_2,F_{D_2})$ in the sense of Connes and Skandalis, \cite{CS84}. In particular we have the identity
\[
[E,F_D] = [E_1,F_{D_1}] \hot_B [E_2,F_{D_2}]
\]
in the KK-group $KK_0(A,C)$.
\end{thm}

% \bibliographystyle{plainmath}
% \bibliography{references}

%% \bibliographystyle{amsalpha-lmp}
%% \bibliography{JK}

\newcommand{\noopsort}[1]{}\def\cprime{$'$}

\end{document}